\documentclass[10pt]{amsart}
\usepackage[utf8]{inputenc}

\usepackage{amssymb,amsthm,amsmath}
\usepackage{mathrsfs}
\usepackage{enumerate}
\usepackage{graphicx,xcolor}
\usepackage{setspace}
\usepackage[hidelinks]{hyperref}
\usepackage{comment}
\usepackage{bbm}
\usepackage{float}
\usepackage{booktabs}
\usepackage{pgfplots}
\usepackage{subfig}

\pgfplotsset{compat=1.18}

\usepackage[paper=a4paper, left=1.2in, right=1.2in, top=1.2in, bottom=1.2in]{geometry}
\newcommand{\ls}{\leqslant}
\newcommand{\gr}{\geqslant}
\newcommand{\E}{\mathbb{E}}

\newcommand{\R}{\mathbb{R}}
\newcommand{\C}{\mathbb{C}}

\usepackage{dsfont}

\numberwithin{equation}{section} % Equations numbered by section

\newtheorem{theorem}{Theorem}[section] % <-- Add [section] here
\newtheorem{lemma}[theorem]{Lemma}
\newtheorem{corollary}[theorem]{Corollary}
\newtheorem{proposition}[theorem]{Proposition}

\theoremstyle{remark}
\newtheorem{remark}[theorem]{Remark}
\newtheorem*{claim}{Claim}
\theoremstyle{definition} % upright body
\newtheorem*{informaltheorem}{Theorem}
\newtheorem{conjecture}[theorem]{Conjecture}

\theoremstyle{definition}
\newtheorem{defn}[theorem]{Definition}

\hypersetup{
    colorlinks,
    linkcolor={black!30!red},
    citecolor={red!60!black},
    urlcolor={blue!80!black}
} 

\title{Sharp inequalities for symmetric polynomials, Hunter’s conjecture, and moments of exponential random variables}

\author{Silouanos Brazitikos}

\author{Christos Pandis}

\address{Department of Mathematics \& Applied Mathematics, University of Crete, Voutes Campus, 70013 Heraklion, Greece}

\email{silouanb@uoc.gr, chrpandis@gmail.com}

\subjclass[2020]{Primary 05E05, 60E15, 26D15; Secondary 15A60, 52A40.}

\begin{document}

\maketitle
\begin{abstract}
We prove Hunter's conjecture on complete homogeneous symmetric polynomials.
For even $n$ and every integer $k\ge1$, we show that under the constraint
$\sum_{i=1}^n a_i^2=1$ the global minimum of the even-degree polynomial
$h_{2k}(a_1,\dots,a_n)$ is attained precisely at the half-plus/half-minus
vector and we compute the optimal value in closed form. The proof combines
algebraic properties of $h_{2k}$ with the probabilistic representation
$k!\,h_k(a)=\mathbb{E}(\sum_{i=1}^n a_iX_i)^k$, where
$X_1,\dots,X_n$ are i.i.d. standard exponential random variables with density
$e^{-x}\mathbf{1}_{x>0}$ and a combinatorial identity. This viewpoint further yields sharp upper and lower
bounds for $\mathbb{E}|\sum_{i=1}^n a_iX_i|^{q}$ under natural constraints on the coefficients, including the spherical
constraint $\sum a_i^2=1$ combined with the non-negative regime $a_i\ge0$, or the centred regime $\sum a_i=0$. Moreover, we determine the exact minimum of $h_{2k}$ on the $\ell_\infty$-sphere 
$S_\infty = \{a \in \mathbb{R}^n : \|a\|_\infty = 1\}$, which yields sharp norm 
comparison inequalities between the matrix norms induced by complete homogeneous 
symmetric polynomials and the classical operator and Schatten norms.
\end{abstract}

\tableofcontents
\section{Introduction}

Complete homogeneous symmetric (CHS) polynomials play a fundamental role in
algebraic combinatorics, representation theory, and the study of moment
inequalities in probability. Let $h_k$ denote the degree-$k$ CHS polynomial in
$n$ real variables $a_1,\dots,a_n$, defined by
\[
h_k(a_1,\dots,a_n)
= \sum_{1 \le i_1 \le \cdots \le i_k \le n} a_{i_1}\cdots a_{i_k}.
\]
We adopt the convention that $h_0(a)=1$. The interplay between the algebraic
structure of $h_k$ and the analytic properties of polynomials on $\R^n$ is
particularly transparent in low degrees. For instance,
\[
h_1(a)=\sum_{i=1}^n a_i, \qquad
h_2(a)=\sum_{i=1}^n a_i^2 + \sum_{1 \le i < j \le n} a_i a_j.
\]
Crucially, the quadratic polynomial admits the sum-of-squares decomposition
\[
h_2(a)
=\frac{1}{2}\sum_{i=1}^n a_i^2
+\frac{1}{2}\Big(\sum_{i=1}^n a_i\Big)^2.
\]
Consequently, $h_2$ is manifestly positive definite. It is natural to inquire
whether higher even-degree CHS polynomials enjoy similar positivity properties.
While positivity fails trivially for odd degrees (since
$h_{2k+1}(-a)=-h_{2k+1}(a)$), Hunter initiated the systematic study of this
phenomenon for even degrees in \cite{hunter1975some}. In \cite{hunter1977positive},
he established that for any integer $k \ge 1$,
\begin{equation}\label{Hunter og}
h_{2k}(a) \ge \frac{1}{k! \cdot 2^k} \left(\sum_{i=1}^n a_i^2\right)^k.
\end{equation}
Equality holds in \eqref{Hunter og} if and only if $k = 1$ and
$\sum_{i=1}^n a_i = 0$. Hunter's positivity theorem has since been revisited and re-proved by several authors using various techniques (see, e.g., \cite{Tao2017Schur, bouthat2023hunter}). 
More recently, Bouthat, Ch\'avez and Garcia \cite{bouthat2023hunter} developed a systematic probabilistic and operator-theoretic framework around Hunter's theorem, interpreting even-degree CHS polynomials as building blocks of ``random vector norms'' on spaces of matrices and surveying many of the existing proofs and generalizations.

In the same work \cite{hunter1977positive}, Hunter also \emph{conjectured} a substantially stronger
statement. Specifically, he conjectured that when $n$ is even, under the
normalization $\sum a_i^2=1$, the global minimum of $h_{2k}$ is attained at the
“half-plus/half-minus’’ vector
\[
\tilde{a}
=\Big(\underbrace{\tfrac{1}{\sqrt{n}}, \dots, \tfrac{1}{\sqrt{n}}}_{n/2},
\underbrace{-\tfrac{1}{\sqrt{n}}, \dots, -\tfrac{1}{\sqrt{n}}}_{n/2}\Big).
\]
Progress on this conjecture has been incremental. Baston \cite{baston1978two}
sharpened Hunter's original bound by adding a correction term depending on
$(\sum a_i)^{2k}$. More recently, Tao \cite{Tao2017Schur} demonstrated that each
$h_{2k}$ is Schur-convex on $\R^n$, which implies that if one fixes
$\sum a_i=1$, then the minimum is attained at the flat vector
$(1/n,\dots,1/n)$. However, the global minimization problem on the sphere
$\mathbb{S}^{n-1}$ has remained open.

A different line of work was recently initiated by Garcia and Vol\v{c}i\v{c} \cite{garcia2025noncommutative}, who introduced noncommutative complete homogeneous symmetric (NCHS) polynomials and proved a noncommutative Hunter-type theorem. They obtained optimal operator-valued lower bounds
\[
H_{2k}(X_1,\ldots,X_n) \succeq \mu_{n,k}\,(X_1^{2k}+\cdots+X_n^{2k})
\]
for Hermitian operators $X_i$, together with explicit sum-of-hermitian-squares representations. In the commutative scalar case this yields new inequalities of the form
\[
h_{2k}(a_1,\ldots,a_n) \ge \frac{\mu_{n,k}}{n^{k-1}} \,\|a\|_2^{2k},
\]
which improve Hunter's original constant whenever $d$ is sufficiently large compared to $n$. Nevertheless, the exact best lower bound for scalar CHS polynomials---that is, the optimal constant and the extremizing configurations on the Euclidean sphere---remained unknown in \cite{garcia2025noncommutative}.

\subsection{Probabilistic and geometric perspectives}

Beyond their algebraic utility, CHS polynomials possess an elegant probabilistic
representation, which will be central to our approach
(see, e.g., \cite{Tao2017Schur,bouthat2023hunter}). If
$X_1,\dots,X_n$ are independent and identically distributed (i.i.d.) standard
exponential random variables (with density $e^{-x}\mathbf{1}_{x>0}$), then for
every integer $k \ge 0$,
\begin{equation}\label{CHS-moment}
k! \, h_k(a) = \E\Big(\sum_{j=1}^n a_j X_j\Big)^k.
\end{equation}
This identity bridges algebraic combinatorics with the study of optimal moment
inequalities: the problem of minimizing $h_{2k}$ on the sphere becomes the
problem of minimizing the even moments of linear combinations of independent
exponential random variables under a fixed variance constraint.

This probabilistic framework also interfaces naturally with the geometry of
convex bodies. Moments of such sums are closely related to the volume of
hyperplane sections of the regular simplex $\Delta_n$. From the geometric
viewpoint (see, for example, \cite{webb1996central}), the constraint
$\sum a_i=0$ is precisely the condition that the hyperplane $a^\perp$ passes
through the centroid of the simplex, and questions about extremal moments in
the zero-sum regime are inherently linked to central sections of $\Delta_n$.

There are thus two particularly compelling reasons to single out the zero-sum
hyperplane
\[
\sum_{i=1}^n a_i = 0.
\]
Geometrically, as just noted, it encodes central sections of the simplex. From
the probabilistic viewpoint, if $X_1,\dots,X_n$ are i.i.d.\ standard exponentials
with $\E X_i=1$, then
\[
\E\Big(\sum_{i=1}^n a_iX_i\Big)
=\Big(\sum_{i=1}^n a_i\Big)\E X_1,
\]
so the linear form $\sum a_i X_i$ is centred if and only if $\sum a_i=0$.
Thus the zero-sum regime is simultaneously the natural geometric setting for
simplex slicing and the natural probabilistic setting for sharp moment
inequalities of centred exponential
distributions.

\subsection{Main contributions}

In this paper, we leverage the probabilistic perspective to provide a complete
description of the extremal behaviour of complete homogeneous symmetric
polynomials under natural constraints. Our analysis is organized into four
regimes, according to the structure of the coefficient vector
$a=(a_1,\dots,a_n)$.

\medskip
\noindent\textbf{1. The unconditional regime (Hunter's conjecture).}
Our first main result is an affirmative resolution of Hunter's conjecture for
all even degrees.

\begin{informaltheorem}[Informal version of Theorem~\ref{Hunter conj}]
Let $n$ be an even integer and $k \ge 1$. For any $a \in \R^n$ with
$\sum a_i^2=1$,
\[
h_{2k}(a) \ge h_{2k}(\tilde a),
\]
where $\tilde a$ is the half-plus/half-minus vector. The explicit value of this
minimum is given by
\[
h_{2k}(\tilde a)
=\frac{(n/2+k-1)!}{k!\,(n/2-1)!\,n^k},
\]
and equality holds if and only if $a$ is a permutation of $\tilde a$.
\end{informaltheorem}
On the upper side, it is clear that the maximizers of $h_{2k}$ under the
constraint $\sum a_i^2=1$ coincide with those in the non-negative regime,
which we describe in detail below.

\medskip
\noindent\textbf{2. The non-negative regime.}
We next analyze the behaviour of moments and CHS polynomials when the
coefficients are constrained to be non-negative, $a_i \ge 0$, with
$\sum a_i^2=1$. For small degrees, we use Schur-type arguments to show that for
$k\le4$ the map
\[
(x_1,\dots,x_n)\mapsto \E\Big(\sum_{j=1}^n \sqrt{x_j}\,X_j\Big)^k
\]
is Schur-concave on $\R_+^n$, yielding sharp two-sided bounds for $h_k$ in terms
of extremal vectors with either one nonzero coordinate or all coordinates
equal. Schur-concavity (or convexity) breaks for $k>4$, thus for higher degrees we turn to an explicit interpolation formula for
$\E\big(\sum a_i X_i\big)^k$ in terms of the coefficients $a_i$ and
Gamma functions. A detailed analysis of this formula shows that, for each fixed
integer $k$ and under the constraint $\sum a_i^2=1$, every minimizer among
non-negative vectors has a very rigid structure: it is supported on a subset of
coordinates on which all entries are equal, and all remaining coordinates are
zero. In other words, all minimizers are of the form
\[
(a_1,\dots,a_n) = (\underbrace{t,\dots,t}_{m\ \text{times}},0,\dots,0),
\]
for some $m\in\{1,\dots,n\}$ and $t>0$ determined by the normalization. The
optimal support size $m$ is characterized via a one-dimensional function. Dually, we prove that all maximizers in the non-negative
regime are vectors with exactly $n-1$ equal coordinates, that is, of the form
\[
(a_1,\dots,a_n) = (s,\underbrace{t,\dots,t}_{n-1\ \text{times}}),
\]
with $t\leq s$ and $(s,t)$ explicitly determined by $k$ and $n$ as a root of an explicit polynomial. These maximizers are, of course, maximizers for the unconditional regime. 

\medskip
\noindent\textbf{3. The centred (geometric) regime.}
Motivated by the simplex slicing problem and the probabilistic setting of
centred random variables, we analyze in detail the case where
\[
\sum_{i=1}^n a_i = 0, \qquad \sum_{i=1}^n a_i^2 = 1.
\]
We prove that for even $n$ the lower bound for $h_{2k}$ under the zero-sum
constraint coincides with the unconditional Hunter bound, and is again attained
at the half-plus/half-minus vector. We also determine the exact maximizers
under the same constraint, which turn out to be vectors with $n-1$ equal
coordinates and one opposite coordinate.

When dealing with the sum of three exponential random variables, we combine these ideas with Fourier-analytic formulas for
moments to obtain sharp upper and lower bounds for nearly all exponents
$q\in(-1,\infty)$. This leads to a complete description of the extremizers for
$\E|\sum a_i X_i|^q$ under the zero-sum constraint.

\medskip
\noindent\textbf{4. Matrix-norm inequalities.}
A further motivation for our work comes from unitarily invariant norms on matrices
induced by complete homogeneous symmetric polynomials. Following
Aguilar, Ch\'avez, Garcia and Vol\v{c}i\v{c} \cite{AguilarKonrad}, given
$A\in M_n(\C)$ with singular values $s_1(A)\ge\cdots\ge s_n(A)\ge0$ and an even
integer $d=2k$, one can define the CHS–norm
\[
\|A\|_{H_d}:=h_d\big(s_1(A),\ldots,s_n(A)\big)^{1/d}.
\]
These norms interpolate between classical Schatten norms, and the authors proved
two-sided comparisons with the operator norm $\|\cdot\|_{\mathrm{op}}$;
see in particular \cite[Theorem~38]{AguilarKonrad}. 
The dependence of their constants on $d$ and $n$ is not optimal for the lower bound, and the
authors explicitly asked for the sharp form of such inequalities.
Our results in Section~6 answer this question, see Theorem~\ref{thm:min-h-2k-linfty}, and lead to the
optimal order of the best constant in the comparison between
$\|\cdot\|_{H_d}$ and $\|\cdot\|_{\mathrm{op}}$.

\subsection*{Organization of the paper}

The rest of the paper is organized as follows: In Section~2, we collect some preliminaries and further develop the necessary background, focusing on Schur-convexity and majorization, which are central to the properties of the complete homogeneous symmetric  polynomials and the Fourier-analytic formulas for moments. In Section~3, we provide an affirmative answer to Hunter's conjecture. Section~4 is devoted to the case of non-negative coefficients. Section~5 addresses the centred case and in Section~6 we study the minimisation of complete homogeneous symmetric polynomials under the constraint $\|a\|_\infty = 1$.

\section{Preliminaries and Background}\label{sec:prelim}

\subsection{Schur-convexity and majorization}

Schur-convexity-type arguments have recently appeared in probabilistic settings
(see, for example, \cite{chasapis2024haagerup,chasapis2023entropies}), leading to
sharp results ranging from moment comparison inequalities to entropy inequalities.
For a concise exposition on majorization and Schur-convexity, we refer to
Chapter~II of \cite{bhatia2013matrix}. We recall here the basic notions that
will be used throughout the paper.

\begin{defn}[Decreasing rearrangement]
Given $x = (x_1, \ldots, x_n)\in\R^n$, we denote by
$x^{\ast} = (x_1^{\ast}, \ldots, x_n^{\ast})$ its decreasing rearrangement, i.e.
\[
x_1^{\ast} \ge x_2^{\ast} \ge \cdots \ge x_n^{\ast}.
\]
\end{defn}

\begin{defn}[Majorization]
For any two vectors $x, y \in \mathbb{R}^n$, we say that $x$ is majorized by $y$,
and write $x \prec y$, if
\[
\sum_{i=1}^n x_i = \sum_{i=1}^n y_i
\quad \text{and} \quad
\sum_{i=1}^k x_i^{\ast} \le \sum_{i=1}^k y_i^{\ast}
\quad \text{for every } k = 1, 2, \ldots, n.
\]
\end{defn}

As a direct consequence, for every vector $a = (a_1, \ldots, a_n) \in \mathbb{R}^n_+$
such that $\sum_{i=1}^n a_i = 1$, we have
\[
\left( \frac{1}{n}, \ldots, \frac{1}{n} \right)
\prec (a_1, \ldots, a_n)
\prec (1, 0, \ldots, 0).
\]
More specifically, if $\sum_{i=1}^n a_i^2 = 1$, then
\begin{equation}\label{maj seq}
\left( \frac{1}{n}, \ldots, \frac{1}{n} \right)
\prec (a_1^2, \ldots, a_n^2)
\prec (1, 0, \ldots, 0).
\end{equation}

\begin{defn}[Schur-convexity/concavity]
A function $f : \mathbb{R}^n \to \mathbb{R}$ is said to be Schur-convex (resp.
Schur-concave) if $x \prec y$ implies $f(x) \le f(y)$ (resp. $f(x) \ge f(y)$).
\end{defn}

A central criterion for establishing the Schur-convexity or Schur-concavity of
a function is due to Schur and Ostrowski.

\begin{theorem}[Schur--Ostrowski]\label{thm:Schur-Ostrowski}
Let $f:\mathbb{R}^n\rightarrow \mathbb{R}$ be a symmetric function with
continuous partial derivatives. Then $f$ is Schur-convex (resp. Schur-concave)
if and only if
\[
(x_i - x_j) \left( \frac{\partial f}{\partial x_i} -
\frac{\partial f}{\partial x_j} \right) \ge 0
\quad (\text{resp. } \le 0)
\]
for all $x \in \mathbb{R}^n$ and for all $1 \le i, j \le n$.
\end{theorem}

\subsection{Complete homogeneous symmetric polynomials}

In the Introduction we defined the complete homogeneous symmetric polynomial $h_k(a_1,\ldots,a_n)$ by
\[
h_k(a_1,a_2,\dots,a_n)
:=\sum_{1\leq i_1 \leq i_2 \leq \cdots \leq i_k\le n}
a_{i_1}a_{i_2}\dots a_{i_k}.
\]
One can also define all the complete homogeneous symmetric polynomials of $n$
variables simultaneously by means of the generating function:
\begin{equation}\label{CHS by mgf means}
\sum_{k=0}^{\infty}h_k(a_1,a_2,\dots,a_n)t^k
= \frac{1}{(1-ta_1)(1-ta_2)\dots(1-ta_n)}.
\end{equation}
As a direct consequence of the generating function representation, we obtain the
following two important properties.

\begin{lemma}[Lemmas~1 and 2 in \cite{hunter1977positive}]\label{Hunter lemma}
If $a\neq b$, then 
\begin{equation}\label{CHS diff}
    h_{k-1}(x,a)-h_{k-1}(x,b)=(a-b)h_{k-2}(x,a,b),
\end{equation}
and
\begin{equation}\label{CHS der}
   \frac{\partial}{\partial x_i}h_{k}(x)=h_{k-1}(x,x_i)
\end{equation}
for every $k\ge1$.
\end{lemma}

Another well-known formula for CHS polynomials is the Lagrange interpolation
formula
\begin{equation}\label{lagr_pol}
h_k(x_1, \ldots, x_n)
= \sum_{i=1}^n \frac{x_i^{n+k - 1}}{\prod_{j\neq i} (x_i - x_j)}.
\end{equation}

We also recall the probabilistic representation already used in the
introduction. Let $X_1, \ldots, X_n$ be i.i.d.\ standard exponential random
variables. Then for any $k \in \mathbb{N}$ we have
\begin{align*}
\mathbb{E}\Big(a_1X_1 + \cdots + a_nX_n\Big)^k
&= \mathbb{E} \left( \sum_{m_1 + \cdots + m_n = k}
\frac{k!}{m_1! \cdots m_n!} a_1^{m_1} \cdots a_n^{m_n} X_1^{m_1} \cdots X_n^{m_n} \right) \\
&= k! \sum_{m_1 + \cdots + m_n = k}
\frac{\mathbb{E}(X_1^{m_1}) \cdots \mathbb{E}(X_n^{m_n})}{m_1! \cdots m_n!}
a_1^{m_1} \cdots a_n^{m_n} \\
&= k! \cdot h_k(a_1, \ldots, a_n),
\end{align*}
where in the last step we used the definition of $h_k$ and the moment identity
$\E(X_i^{m_i})=m_i!$. This is exactly the representation
\eqref{CHS-moment}.

In \cite{hunter1977positive}, Hunter was the first to show that even-degree CHS
polynomials are positive definite.

\begin{theorem}[Hunter]\label{thm:hunter-positivity}
Let $n,k$ be non-negative integers. Then $h_{2k}(x_1,\ldots,x_n)$ is a positive
definite function on $\mathbb{R}^n$, i.e.\ $h_{2k}(x_1,\ldots,x_n)>0$ for all
$x\neq0$.
\end{theorem}

Tao established the positive definiteness and Schur-convexity of the CHS polynomials in
\cite{Tao2017Schur}.

\begin{theorem}[Tao]\label{thm:Tao}
Let $n,k$ be non-negative integers. Then, for any $x \in \mathbb{R}^n$, the
following hold.
\begin{enumerate}
    \item[(i)] \textbf{Positive definiteness:}
    $h_{2k}(x)\ge0$, with equality if and only if $x=0$.
    \item[(ii)] \textbf{Schur-convexity:}
    $h_{2k}(x)\le h_{2k}(y)$ whenever $x\prec y$. Moreover, equality holds if
    and only if $x$ is a permutation of $y$.
    \item[(iii)] \textbf{Schur--Ostrowski criterion:}
    For every $1\le i<j\le n$,
    \[
    (x_i-x_j)\Big(\frac{\partial}{\partial x_i}-\frac{\partial}{\partial x_j}\Big)
    h_{2k}(x)\ge0,
    \]
    with strict inequality unless $x_i=x_j$.
\end{enumerate}
\end{theorem}

Hunter’s positivity theorem for CHS polynomials has been rediscovered and proved
many times; for additional proofs and extensions we refer to
\cite{bouthat2023hunter} and the references therein.

\subsection{Power-sum symmetric polynomials}

The power-sum symmetric polynomial of degree $m$ in the variables
$x_1,\ldots,x_n$ is defined, for $m\in \mathbb{N}$, by
\[
p_m(x_1,\ldots,x_n)=x_1^m+x_2^m+\cdots+x_n^m,
\]
often written $p_m(\mathbf{x})$ or simply $p_m$ when the variables are clear
from context. The CHS polynomials and the power-sum polynomials are connected by
the following well-known identity (see, for instance,
\cite{macdonald1998symmetric,stanley1999enumerative}):
\begin{equation}\label{CHP-PSP}
h_k(\mathbf{x})
= \sum_{\substack{m_1 + 2m_2 + \cdots + km_k = k \\ m_1\ge 0, \ldots, m_k\ge 0}}
\prod_{i=1}^k \frac{p_i(\mathbf{x})^{m_i}}{m_i ! \, i^{m_i}}.
\end{equation}
All coefficients in this expansion are nonnegative. This combinatorial identity
will play a crucial role in our proof of Hunter’s conjecture, as it allows us to
compare $h_{2k}$ on different vectors by comparing only the corresponding power
sums.

\subsection{Fourier-analytic formulas}

Fourier-analytic formulas for moments and negative moments of random vectors
have played a crucial role in the study of various slicing problems in convex
geometry; see, for example,
\cite{chasapis2021ball,chasapis2022slicing,chasapis2024haagerup}. We recall
here the classical formulas that we shall need.

\begin{lemma}[Lemma~3 in \cite{gorin1991generalizations}]
Let $X$ be a random vector in $\mathbb{R}^d$ and let $p\in(0,d)$. Then
\[
\mathbb{E}\|X\|^{-p}
= b_{p,d}\int_{\mathbb{R}^d}\phi_X(t) \|t\|^{p-d}\,dt,
\]
provided that the right-hand side integral exists, where
$\phi_X(t)= \mathbb{E}e^{i\langle t , X \rangle}$ is the characteristic function
of $X$, $\|\cdot\|$ is the Euclidean norm on $\mathbb{R}^d$, and
\[
b_{p,d}
= 2^{-p}\pi^{-d/2}\frac{\Gamma\big((d-p)/2\big)}{\Gamma(p/2)}.
\]
\end{lemma}

There are also Fourier-type formulas for positive moments.

\begin{lemma}[Lemmas~2.3 and 2.4 in \cite{haagerup1981best}]
Let
\[
C_p = \frac{2}{\pi} \Gamma(1 + p) \sin\left( \frac{p\pi}{2} \right).
\]
For a real-valued random variable $X$ with characteristic function
$\phi_X(t) = \mathbb{E}(e^{itX})$, we have, for $p \in (0,2)$,
\begin{equation}\label{Fourier (0,2)}
\mathbb{E}|X|^p = C_p \int_{0}^{\infty} \frac{1 - \Re(\phi_X(t))}{t^{p+1}} \, dt.
\end{equation}
For $p \in (2,4)$, assuming $\mathbb{E}(X^4) < \infty$, it holds that
\begin{equation}\label{Fourier (2,4)}
\mathbb{E}|X|^p = -C_p \int_{0}^{\infty}
\left( \Re(\phi_X(t)) - 1 + \frac{1}{2} \mathbb{E}(X^2) t^2 \right)
t^{-(p+1)} \, dt.
\end{equation}
\end{lemma}

Using the method introduced in \cite{haagerup1981best} to prove the above lemmas  we can actually prove the following:
\begin{lemma} \label{Fourier (4,6)}
    Let $X$ be a real-valued  random variable that satisfies $\mathbb{E}X^6<\infty$. For $p\in(4,6)$ we have 
    \[
     \mathbb{E}|X|^p=C_p \int_{0}^{\infty}(-\Re(\phi_X(t))+1-\frac{1}{2}\mathbb{E}(X^2)t^2+\frac{1}{4!}\mathbb{E}(X^4)t^4)t^{-(p+1)}dt,
    \]
    where $C_p$ is the previous constant.
\end{lemma}
\begin{proof} 

Let $x\in \mathbb{R}$, we will compute 
\[
M:=\int_{0}^{\infty}(-\cos(xt)+1-\frac{1}{2!}x^2t^2+\frac{1}{4!}x^4t^4)t^{-p-1} \, dx.
\]
Notice that  $-\cos(u)+1-\frac{1}{2}u^2+\frac{1}{4!}u^4>0$ for  $u>0$. Since,  $\cos(t)=1-\frac{1}{2}t^2+\frac{1}{4!}t^4+ O(t^6)$ for $t\to 0$ and $\cos(t)=1-\frac{1}{2}t^2+\frac{1}{4!}t^4+ O(t^4)$ for $t\to \infty$ we see that $M$ is finite. 
Using the substitution $u=|x|t$ and integrating by parts, we get
  \begin{align*}
M 
&= |x|^{p}\int_{0}^{\infty}
(-\cos u+1-\tfrac{1}{2!}u^{2}+\tfrac{1}{4!}u^{4})u^{-p-1}\, du \\[4pt]
&= \frac{|x|^{p}}{p(p-1)(p-2)(p-3)}
\int_{0}^{\infty}
(1-\cos u)u^{-p+3}\, du \\[4pt]
&= \frac{|x|^{p}}{p(p-1)(p-2)(p-3)}\,\frac{1}{C_{p-4}}
= |x|^{p}\,\frac{1}{C_{p}} .
\end{align*}

In the last steps we used the facts that $p(p-1)C_{p-2}=-C_p$,  $0<p-4<2$ and 
\[
\int_{0}^{\infty}(1-\cos u)u^{-q-1}=\frac{1}{C_q},
\]
for $0<q<2$ (see \cite{haagerup1981best}).
Thus,
 \[
 |x|^p=C_p  \int_{0}^{\infty}(-\cos(xt)+1-\frac{1}{2}x^2t^2+\frac{1}{4!}x^4t^4)t^{-p-1}
 \]

 The result follows from the fact $\Re(\phi_X(t))=\Re(\mathbb{E}(e^{itx}))=\mathbb{E}(\cos(tX))$ combined with Fubini's Theorem. 
\end{proof}

\subsection{Weighted sums of exponential random variables}

We next recall an explicit representation of the density of weighted sums of
independent exponentials. It is a folklore result (see, e.g.,
\cite{brzezinski2013volume}) that the density of the linear combination
$a_1 X_1 + \cdots + a_n X_n$, denoted by $G$, where
$X_1, \ldots, X_n$ are i.i.d.\ standard exponential random variables, is given
by
\begin{align*}
G(t)
&= \sum_{\substack{j=1 \\ a_j > 0}}^{n} 
\frac{1}{a_j} 
\prod_{\substack{k=1 \\ k \neq j}}^{n} 
\frac{a_j}{a_j - a_k} 
e^{-t/a_j} \, \mathbbm{1}_{[0, \infty)}(t) \\
&= - \sum_{\substack{j=1 \\ a_j < 0}}^{n} 
\frac{1}{a_j} 
\prod_{\substack{k=1 \\ k \neq j}}^{n} 
\frac{a_j}{a_j - a_k} 
e^{-t/a_j} \, \mathbbm{1}_{(-\infty, 0]}(t),
\end{align*}
that is, for $t\neq 0$,
\begin{equation}\label{density of a_iX_i}
G(t)
= \sum_{\substack{j=1\\a_j>0}}^{n}
 \frac{1}{a_j} \prod_{\substack{k=1\\k \neq j}}^{n}
 \frac{a_j}{a_j - a_k} e^{-t/a_j} \mathbbm{1}_{[0,\infty)}(t)
- \sum_{\substack{j=1\\a_j<0}}^{n}
 \frac{1}{a_j} \prod_{\substack{k=1\\k \neq j}}^{n}
 \frac{a_j}{a_j - a_k} e^{-t/a_j} \mathbbm{1}_{(-\infty,0]}(t),
\end{equation}
and
\begin{equation}\label{density of a_iX_i at 0}
G(0)
=\frac{1}{2}\left(
\sum_{\substack{j=1\\a_j>0}}^{n} \frac{1}{a_j}
\prod_{\substack{k=1\\k \neq j}}^{n} \frac{a_j}{a_j - a_k} 
- \sum_{\substack{j=1\\a_j<0}}^{n} \frac{1}{a_j}
\prod_{\substack{k=1\\k \neq j}}^{n} \frac{a_j}{a_j - a_k} \right).
\end{equation}

This, in turn, implies the following interpolation formula:
\begin{equation}\label{interpol}
\mathbb{E}\left|\sum_{j=1}^n a_j X_j \right|^q
= \Gamma(1+q)\cdot \left(
\sum_{j=1}^n |a_j|^q \prod_{i \neq j} \frac{a_j}{a_j - a_i} \right),
\end{equation}
which remains valid for all $q+1 > 0$, and also for $q+1<0$ provided that
$q+1$ is not an integer. Here $\Gamma$ denotes the Euler gamma function,
defined by
\[
\Gamma(a)=\int_{0}^{\infty}t^{a-1}e^{-t}\,dt \quad \text{for $\Re(a)>0$},
\]
and extended to all $a<0$ except at its poles $\{0,-1,-2,\ldots\}$ by the
recurrence $\Gamma(a)=\Gamma(a+1)/a$.
\begin{comment}
\begin{remark}
    Notice also that, if $\sum_{i=1}^n a_i=0$, then
    \[
    \lim_{q\to1}\frac{1-q}{2}\E\left|\sum_{j=1}^n a_jX_j \right|^{-q}
    =\frac{1}{2} \sum_{j=1}^n |a_j|^{-1} \prod_{i \neq j} \frac{a_j}{a_j - a_i}
    =\frac{(n-1)!}{\sqrt{n+1}}\operatorname{vol}_{n-1}\left(\Delta^n\cap a^{\perp}\right),
    \]
    where in the last equality we used Webb's interpolation formula
    \cite{webb1996central}. Thus we recover the probabilistic representation of
    central simplex sections used, for instance, in
    \cite{melbourne2025simplex}.
\end{remark}
\end{comment}
\subsection{Palindromic and anti-palindromic polynomials}

Finally, we record a simple algebraic notion that will be used in some auxiliary
arguments.

\begin{defn}
Given a polynomial $P(x)=a_0+a_1x+\ldots+a_nx^n$, we say that it is
\emph{palindromic} if $a_i=a_{n-i}$ for all $i=0,1,\ldots,n$, i.e.\ if its
coefficients, when the polynomial is written in the order of ascending or
descending powers, form a palindrome.

Similarly, a polynomial $P$ of degree $n$ is called \emph{anti-palindromic} if
$a_i=-a_{n-i}$ for all $i=0,1,\ldots,n$.
\end{defn}

An immediate property of an anti-palindromic polynomial $P(x)$ is that $x = 1$
is always a root. 

\section{A Proof of Hunter's Conjecture}\label{sec:hunter}

In this section we prove our first main result, which gives a complete solution
to Hunter's conjecture in the scalar case. We begin with the precise
description of the extremizers for $h_4$, and then proceed to all even degrees.

\begin{proposition}\label{Hunter h_4} 
 Let $a_1, \ldots, a_n$ be real numbers such that $\sum_{i=1}^n a_i^2 = 1$. Then $h_4$ attains its maximum at the vector 
\[
\overline{a} = \left( \frac{1}{\sqrt{n}}, \ldots, \frac{1}{\sqrt{n}} \right),
\]
while it attains its minimum at the “half-plus/half-minus” vector
\[
\tilde{a} = \left(\underbrace{\tfrac{1}{\sqrt{n}}, \dots, \tfrac{1}{\sqrt{n}}}_{n/2}, \underbrace{-\tfrac{1}{\sqrt{n}}, \dots, -\tfrac{1}{\sqrt{n}}}_{n/2}\right)
\]
when $n$ is even, and at a vector of the form
\[
\left( \underbrace{a, \ldots, a}_{\frac{n-1}{2}}, \underbrace{b, \ldots, b}_{\frac{n+1}{2}} \right),
\]
when $n$ is odd. Here $a$ appears $\frac{n-1}{2}$ times, $b$ appears $\frac{n+1}{2}$ times, and $\frac{a}{b}$ minimizes the function
\[
\frac{x^2 + 1 + \left( \frac{n+1}{2}x + \frac{n+3}{2} \right)^2}{\frac{n-1}{2}x^2 + \frac{n+1}{2}}.
\]
\end{proposition}

Our main theorem settles Hunter's conjecture for all even degrees.

\begin{theorem}\label{Hunter conj}
Let $n$ be an even integer, and let $a_1, \ldots, a_n$ be real numbers such that $\sum_{i=1}^n a_i^2 = 1$. Then, for every integer $r\ge1$,
\[
h_{2r}(a_1, \ldots, a_n) \geq \frac{\big(\frac{n}{2} + r - 1\big)!}{r! \cdot \big(\frac{n}{2} - 1\big)! \cdot n^r},
\]
and this inequality is sharp, with equality achieved if and only if
$(a_1,\dots,a_n)$ is a permutation of the half-plus/half-minus vector $\tilde{a}$.
\end{theorem}

One may now ask about the maximum of $h_{2r}$ under the normalization condition $\sum_i a_i^2 = 1$. To this end, we observe that a straightforward application of the triangle inequality in identity~\eqref{CHS-moment} reduces the problem to the case where $a_i \geq 0$ and $\sum_i a_i^2 = 1$. We will elaborate on this reduction later when treating the case involving positive coefficients.

Before proceeding to the proof of the conjecture, we present a useful Proposition suggesting that the extrema of $h_{2k}$ are attained under specific structural conditions.
\begin{proposition} \label{Lag integers unconditional}
    Let $n\gr1$ and $d>3$  be a non-negative even  integer. Then the extrema of $h_d$ on the unit sphere $\mathbb{S}^{n-1}$ are of the form $$\boldsymbol{x}=\left(\underbrace{a, \ldots, a}_{\gamma_1 }, \underbrace{b, \ldots, b}_{\gamma_2 },\underbrace{c,\ldots,c}_{\gamma_3 
 }\right).$$ Here, $ a $ appears $ \gamma_1 $ times, $ b $ appears $\gamma_2 $ times, and $ c $ appears $ \gamma_3 $ times, subject to the constraints $\gamma_1 a^2 + \gamma_2 b^2 + \gamma_3 c^2 = 1$ and $\gamma_1 + \gamma_2 + \gamma_3 = n$. 
\end{proposition}
\begin{proof}
    The unit sphere in $\mathbb{R}^n$ is compact. Therefore, $h_d$ must attain its extrema $\boldsymbol{x}$ on the unit sphere in $\mathbb{R}^n$. The method of Lagrange multipliers ensures that if $\boldsymbol{x}$ is a extrema, there exists $\lambda$  such that 
\begin{equation}
    \frac{\partial h_{d}(\boldsymbol{x})}{\partial x_i} +2\lambda x_i =0
\end{equation}
    for each $i=1,2,\dots,n$. We multiply  by $ x_i $ and sum over all $ i $ to obtain 
    \begin{equation}\label{Lag 2}
    \sum_{i=1}^{n}x_i\frac{\partial h_d(\boldsymbol{x})}{\partial x_i} +2\lambda  =0.
    \end{equation}
    From Euler’s homogeneous function theorem, equation (\ref{Lag 2}) becomes $$dh_{d}(\boldsymbol{x})+2\lambda=0.$$
    Substituting and using the differentiating property  of the CHS polynomials \ref{CHS der}, we obtain 
    \begin{equation}\label{h_{d-1}=dx_ih_d}
    h_{d-1}(\boldsymbol{x}, x_i) = d x_i h_d(\boldsymbol{x})
    \end{equation}
for each $i = 1, 2, \dots, n$.
 
The vector with all coordinates equal  satisfies equation~\eqref{h_{d-1}=dx_ih_d}. Thus, we may assume that there exist coordinates $ x_i \neq x_j $. Applying equation~\eqref{h_{d-1}=dx_ih_d} to $x_i$ and $x_j$, and subtracting the results, combined with the difference property, suggests that
\[
(x_i-x_j)h_{d-2}(\boldsymbol{x},x_i,x_j)=h_{d-1}(\boldsymbol{x},x_i)-h_{d-1}(\boldsymbol{x},x_j)=d(x_i-x_j)h_{d}(\boldsymbol{x}).
\]
Then,
\begin{equation}\label{dh_d=h_{d-2}}
h_d(\boldsymbol{x})=\frac{h_{d-2}(\boldsymbol{x},x_i,x_j)}{d}.
\end{equation}
Assume that there exists a third distinct coordinate $x_k \neq x_i, x_j$. By applying relation (\ref{dh_d=h_{d-2}}) once again for $x_i$ and $x_k$, and subtracting as before, we obtain
\begin{equation}\label{Lag d-3}
h_{d-3}(\boldsymbol{x}, x_i, x_j,x_k) = 0.
\end{equation}
If we further consider,  $x_l\neq x_i,x_j,x_k$ in the same manner we obtain
$$
h_{d-4}(\boldsymbol{x}, x_i, x_j, x_k, x_l) = 0.
$$
For $d\gr4$, since $d - 4$ is even, the positivity of the even degree CHS polynomials  leads to a contradiction.
\end{proof}
We proceed by deriving sharp bounds for $h_4$ through appropriate estimates of its extrema. This will play a crucial role as the inductive step.
\begin{proof}[Proof of Proposition \ref{Hunter h_4}]
For the maximum notice that 
\[
h_4(a_1,\ldots,a_n)=\frac{1}{4!}\E\left(\sum_{i=1}^na_iX_i \right)^4\ls \frac{1}{4!}\E\left(\sum_{i=1}^n|a_i|X_i \right)^4\ls h_4\left(\frac{1}{\sqrt{n}},\ldots,\frac{1}{\sqrt{n}} \right),
\]
where we used Corollary~\ref{Cor Gamma pos k} for $k=4$.

For the minimum, we shall use the method of Lagrange multipliers as in proof of Proposition \ref{Lag integers unconditional} to bound the extrema of $h_4(a_1,\ldots,a_n)$ which exist since the domain is compact. We are searching for all $\boldsymbol{x}=(x_1,\ldots,x_n)$ and the real number $\lambda$. By following the preceding argument verbatim, we find that Relation (\ref{Lag d-3}) asserts that
 $$h_1(\boldsymbol{x},x_i,x_j,x_k)=0$$
or equivalently, if we set $S := \sum_{i=1}^n x_i$, 
\begin{equation}\label{S+a+b+c=0}
 S + x_i + x_j + x_k = 0.
 \end{equation}
We also obtained the following identity (\ref{dh_d=h_{d-2}}): $$  4\cdot h_4(\boldsymbol{x})=h_2(\boldsymbol{x},x_i,x_j)=\frac{1}{2}\left\{\sum_{m=1}^n x_m^2+x_i^2+x_j^2+ \left( S+x_i+x_j\right)^2 \right \}=\frac{1}{2}\left(1+x_i^2+x_j^2+x_k^2 \right).$$ As established in Proposition \ref{Lag integers unconditional}, the extrema are attained under specific structural conditions, that is $$\boldsymbol{x}=\left(\underbrace{a, \ldots, a}_{\gamma_1 }, \underbrace{b, \ldots, b}_{\gamma_2 },\underbrace{c,\ldots,c}_{\gamma_3 
}\right),$$ where $\gamma_1a^2+\gamma_2b^2+\gamma_3c^2=1$. For the moment, we assume that the parameters $a$, $b$, and $c$ are all distinct.

Thus, relation (\ref{S+a+b+c=0}) suggests than it suffices to lower bound for $a^2+b^2+c^2$ under the conditions 
$$\begin{cases}
(\gamma_1+1)a+(\gamma_2+1)b+(\gamma_3+1)c=0&\\
\gamma_1 a^2+\gamma_2b^2+\gamma_3 c^2=1
\end{cases}$$

with $\gamma_1,\gamma_2,\gamma_3\gr1$ and $\gamma_1+\gamma_2+\gamma_3=n$.

A direct computation shows that $h_4(\tilde{a}) = \frac{1}{8} + \frac{1}{4n}$. Thus, it remains to prove the inequality
\[
 a^2+b^2+c^2 \gr \frac{2}{n-1}
\] 

Without loss of generality, we may assume that $c \neq 0$. We write this as
\[ 
a^2+b^2+c^2=\frac{(a/c)^2+(b/c)^2+1}{\gamma_1 (a/c)^2+\gamma_2(b/c)^2+\gamma_3}.
\]
By setting $x:=\frac{a}{c}$ and $y:=\frac{b}{c}$, and using the fact that $(\gamma_1+1)x+(\gamma_2+1)y+(\gamma_3+1)=0$, we reduce the bound to the following quadratic inequality:
\begin{align*}
&x^2 \left[ (\gamma_2 + 1)^2 \left(1 - \frac{2\gamma_1}{n - 1} \right)
        + (\gamma_1 + 1)^2 \left(1 - \frac{2\gamma_2}{n - 1} \right) \right]  + 2x (\gamma_1 + 1)(\gamma_3 + 1) \left(1 - \frac{2\gamma_2}{n - 1} \right) \\
&\quad + (\gamma_3 + 1)^2 \left(1 - \frac{2\gamma_2}{n - 1} \right)
        + (\gamma_2 + 1)^2 \left(1 - \frac{2\gamma_3}{n - 1} \right) \gr 0
\end{align*}
Setting now $d_i:=\gamma_i-1\geq 0$, we notice that the coefficient in front of $x^2$ is non-negative, since 
$$
(d_2+2)^2(d_2+d_3-d_1)+(d_1+2)^2(d_1+d_3-d_2)=(d_1-d_2)^2(d_1+d_2+4)+d_3\left[ (d_2+2)^2+(d_1+2)^2 \right] \gr0.
$$
The discriminant $\Delta$, equals
\begin{align*}
-4(\gamma_2 + 1)^2 \cdot \Big[ \,
    &(\gamma_3 + 1)^2 \left(1 - \frac{2\gamma_2}{n - 1} \right)\left(1 - \frac{2\gamma_1}{n - 1} \right) 
    +(\gamma_2 + 1)^2 \left(1 - \frac{2\gamma_1}{n - 1} \right)\left(1 - \frac{2\gamma_3}{n - 1} \right) \\
    +\,  &(\gamma_1 + 1)^2 \left(1 - \frac{2\gamma_2}{n - 1} \right)\left(1 - \frac{2\gamma_3}{n - 1} \right)
\Big].
\end{align*}
We will prove that $\Delta$ is non-positive. Due to symmetry, we may assume that $\gamma_1\geq \gamma_2\geq \gamma_3.$
Then substituting $n=\gamma_1+\gamma_2+\gamma_3$, it suffices to prove that 
$$\sum (d_1+2)^2(d_1+d_2-d_3)(d_1+d_3-d_2)\geq 0.$$
If $d_3=0$, then the inequality can be rewritten in the form 
$$(d_1 - d_2)^2 (d_1^2 + 2d_1d_2+d_2^2+4d_1+4d_2- 4 ) \geq 0,$$
which is true.\\
If $d_3=1$, then the inequality is equivalent to 
$$(d_1^2-d_2^2)^2+(d_1+d_2)(d_1-d_2)^2+(2d_1^3-10d_1^2+4d_1+1+17d_1d_2)+(2d_2^3-10d_2^2+4d_2+1+17d_1d_2)\geq 0,$$ which holds.\\
Finally, if $d_1,d_2,d_3\geq 2$, we rewrite the inequality as 
$$\sum (d_1+2)^2(d_1-d_3)(d_1-d_2)+\sum (d_1+2)^2d_3(d_1-d_3)+\sum (d_1+2)^2d_2d_3\geq 0.$$
The last sum is clearly non-negative. For the first sum notice that it can be expressed as $$(d_1-d_2)\left[(d_1+2)^2(d_1-d_2)-(d_2+2)^2(d_2-d_3) \right]+(d_3+2)^2(d_3-d_2)(d_3-d_1)\gr0$$ since $d_1\gr d_2\gr d_3$.  For the second one, after collecting the same terms, equals to 
$$\sum (d_1-d_3)^2(d_1d_3-4)\geq 0,$$ which is again true, since $d_1,d_2,d_3\geq 2$.

In the case where $\boldsymbol{x}$ has exactly two distinct coordinates, equation (\ref{Lag d-3}) does not hold. Without loss of generality assume $b\neq0$. In this  case $$\boldsymbol{x}=(\underbrace{a, \ldots, a}_{\gamma_1 \text{ times}},\underbrace{b, \ldots, b}_{\gamma_2 \text{ times}})$$  and $\gamma_1a^2+\gamma_2b^2=1$ holds. In this case, from relation~\eqref{dh_d=h_{d-2}}, we need to lower bound
\[
a^2 + b^2 + \left( (\gamma_1 + 1)a + (\gamma_2 + 1)b \right)^2.
\]
Then we will find the best constant $\frac{2}{n-1}\geq c\geq\frac{2}{n}$, such that the inequality 
\begin{equation} \label{sim}  
a^2+b^2+\left [ (\gamma_1+1)a+(\gamma_2+1)b \right]^2\gr c(\gamma_1a^2+\gamma_2b^2)
\end{equation}
holds for all $a,b$. This can be equivalently expressed as 
\[
x^2 \left[1+(\gamma_1+1)^2-c\gamma_1 \right]+2x(\gamma_1+1)(\gamma_2+1)+(\gamma_2+1)^2+1-c\gamma_2\gr0.
\]
Notice that since $c\leq 1$ we have 
\[
1+(\gamma_1+1)^2-c\gamma_1\gr0
\]
and that
\[
\frac{\Delta}{4}=(\gamma_1+1)^2(\gamma_2+1)^2-\left(1+(\gamma_1+1)^2-c\gamma_1\right)\left(1+(\gamma_2+1)^2-c\gamma_2\right).
\]
We use the fact that $\gamma_1+\gamma_2=n$, to write the last one as a function of $\gamma_1$. 
The derivative of this function with respect to $\gamma_1$ equals to
$$(-c^2 + c n + 4 c + 2) (n - 2\gamma_1).$$
The first parenthesis is of course non-negative, therefore the function is increasing for $\gamma_1\leq n/2$ and decreasing for
$\gamma_1\geq n/2$. If $n$ is even, then it takes its maximum for $\gamma_1=n/2$. The maximum equals to 
$$\frac{1}{4}(-2 + c n) (6 + 4 n - c n + n^2),$$ which is non-positive for $c\leq\frac{2}{n}$. Therefore, for even $n$ we have that 
\begin{equation} \label{sim}  
a^2+b^2+\left [ (\gamma_1+1)a+(\gamma_2+1)b \right]^2\gr \frac{2}{n}(\gamma_1a^2+\gamma_2b^2)
\end{equation}
and the equality holds when $\gamma_1=\gamma_2=n/2$ and $a=-b$.

In the case where $n$ is odd, $\gamma_1$ cannot be equal to $n/2$, therefore the function takes its maximum for $\gamma_1=\frac{n-1}{2}$ (or $\gamma_1=\frac{n+1}{2}$). In the first case, the discriminant is equal to 
$$\frac{1}{4}\left(c^2 (1-n^2)+ c (-4 + 7 n + 4 n^2 + n^3) - 2 (7 + 4 n + n^2) \right).$$
The last one is non-positive if and only if $c\leq \rho_1(n)$ or $c\geq\rho_2(n)$. However, $\rho_2(n)\geq\frac{2}{n-1}$, therefore, the largest value that $c$ can take is 
$$c=\rho_1(n)=\frac{-4 + 7 n + 4 n^2 + n^3 - (n+3)\sqrt{8 - 8 n + n^2 + 2 n^3 + n^4} }{2 n^2-2}.$$
Note that $$\rho_1(n)\sim\frac{2}{n}$$ as $n\to +\infty$. We conclude that for $n$ odd the inequality \begin{equation} \label{sim}  
a^2+b^2+\left [ (\gamma_1+1)a+(\gamma_2+1)b \right]^2\gr \rho_1(n)(\gamma_1a^2+\gamma_2b^2)
\end{equation}
holds, and we have equality when $\gamma_1=(n-1)/2$, $\gamma_2=(n+1)/2$ and $\frac{a}{b}=x$, where $x$ is the minimum value of the function $$\frac{x^2+1+\left(\frac{n+1}{2}x+\frac{n+3}{2}\right)^2}{\frac{n-1}{2}x^2+\frac{n+1}{2}}.$$
\end{proof}
We now proceed with the proof of Hunter's  conjecture.
\begin{proof}[Proof of Theorem \ref{Hunter conj}]
We will prove, by induction on $k$, that every extremum $\boldsymbol{x}$ of $h_{2k}$ on the sphere $\mathbb{S}^{n-1}$, when $n$ is even, satisfies
\[
  h_{2k}(\boldsymbol{x})\gr  h_{2k}(\tilde{a})=\frac{(n/2+k-1)!}{k!\cdot(n/2-1)!\cdot n^k}
\]
We have already established the cases $k = 1$ and $k = 2$. Now, assume the statement holds for $k - 1$ and that the extrema are of the form $$\boldsymbol{x}=\left(\underbrace{a, \ldots, a}_{\gamma_1 }, \underbrace{b, \ldots, b}_{\gamma_2 },\underbrace{c,\ldots,c}_{\gamma_3 
 }\right),$$ where $\gamma_1a^2+\gamma_2b^2+\gamma_3c^2=1$.   

 Due to the symmetry, we can assume that $a>|b|>|c|$.

If $\gamma_1 \gr \gamma_2+\gamma_3$ then $h_{2k}(\boldsymbol{x})\gr  h_{2k}(\tilde{a})$. Indeed, in this case we have that $$p_{2m+1}(\boldsymbol{x})=\gamma_1a^{2m+1}+\gamma_2b^{2m+1}+\gamma_3c^{2m+1}\gr 0= p_{2m+1}(\tilde{a})$$ and using the Power-Mean inequality
$$\left(\frac{\gamma_1a^{2m}+\gamma_2b^{2m}+\gamma_3c^{2m}}{\gamma_1+\gamma_2+\gamma_3}\right)^{1/m}\gr\frac{\gamma_1a^{2}+\gamma_2b^{2}+\gamma_3c^{2}}{\gamma_1+\gamma_2+\gamma_3}=\frac{1}{n},$$
which can be written as 
$$p_{2m}(\boldsymbol{x})\gr p_{2m}(\tilde{a}).$$
Then, from identity~\eqref{CHP-PSP}, which expresses the CHS polynomial solely in terms of the power-sum polynomials, we conclude the desired inequality.

If $\gamma_1\ls \gamma_2+\gamma_3$, we proceed using the already established relation~\eqref{dh_d=h_{d-2}}
\begin{align*}
h_{2k}(\boldsymbol{x})&=\frac{h_{2k-2}(a[\gamma_1+1],b[\gamma_2+1])}{2k}\\
&=\frac{(1+a^2+b^2)^{k-1}}{2k}h_{2(k-1)}\left(\frac{\boldsymbol{x}}{\sqrt{1+a^2+b^2}},\frac{a} {\sqrt{1+a^2+b^2}},\frac{b}{\sqrt{1+a^2+b^2}}\right).
\end{align*}
In this case, we obtain that $$a^2+b^2\gr \frac{2}{n},$$ which helps us to complete the induction.  Assume now that the extrema is of the form $$\boldsymbol{x}=(\underbrace{a, \ldots, a}_{\gamma_1 \text{ times}},\underbrace{b, \ldots, b}_{\gamma_2 \text{ times}})$$ where $a,b$ are distinct and appear $\gamma_1$ and $\gamma_2$ times respectively and thus also $\gamma_1 a^2+\gamma_2b^2=1$. Setting $b=c$ in the argument above which helps us to complete the induction.

\end{proof}
\section{The Non-Negative Coefficients Case}\label{sec:positive}

In this section we study the extremal behaviour of moments and complete
homogeneous symmetric polynomials when the coefficients are constrained to be
non-negative. Throughout we assume $a_i\ge0$ and $\sum_{i=1}^n a_i^2=1$.

For positive integer moments up to order four we have the following
Schur-concavity result.

\begin{theorem} \label{Schur-gamma}
Let  $X_1,X_2,\ldots$ be independent and identically  distributed standard exponential random variables. For any positive integer $k \leq 4$ and $n\in\mathbb{N}$, the function
\[
(x_1,\ldots,x_n) \mapsto \E\left(\sum_{j=1}^n \sqrt{x_j}X_j\right)^{k}
\]
is Schur-concave on $\mathbb{R}^n_+$.
\end{theorem}
Note that for $k>4$ Schur-concavity or Schur-convexity breaks. 

As an immediate corollary, we obtain two-sided moment bounds in terms of the
extreme non-negative configurations.

\begin{corollary}\label{Cor Gamma pos k}
For  $X_1,X_2,\ldots$ i.i.d standard exponential random variables. For any positive integer $k \leq 4$ and $n\in\mathbb{N}$,
\[
\mathbb{E}X_1^k \leq \mathbb{E}\left( \sum_{j=1}^{n} a_j X_j \right)^k \leq \mathbb{E} \left( \frac{X_1 + \cdots + X_n}{\sqrt{n}} \right)^k.
\]
\end{corollary}
\begin{remark}
 By a similar argument, Theorem~\ref{Schur-gamma} remains valid when the standard exponential random variables are replaced with $\mathrm{Gamma}(\gamma)$ random variables, for any positive integer $k \leq 2\gamma + 2$.

\end{remark}
Fot our next results we need the following definition.
\begin{defn}
Let \( X_1, X_2, \ldots, X_n \) be i.i.d.\ standard exponential random variables. For a real number \( q \), we define
\[
\rho(1, q) := \mathbb{E}[X_1^q], \quad \rho(2, q) := \mathbb{E} \left( \frac{X_1 + X_2}{\sqrt{2}} \right)^q,
\]
and for general \( n \in \mathbb{N} \),
\[
\rho(n, q) := \mathbb{E} \left( \frac{X_1 + \cdots + X_n}{\sqrt{n}} \right)^q.
\]
\end{defn}

\begin{theorem}\label{pos k > n+2}
Let $k$ be a non-negative integer, and let $a_1, \ldots, a_n$ be non-negative real numbers such that $\sum_{i=1}^n a_i^2 = 1$. If $X_1, \ldots, X_n$ are i.i.d.\ standard exponential random variables, then
\[
\mathbb{E}\left( a_1 X_1 + \cdots + a_n X_n \right)^k \ge \min\{ \rho(1, k), \ldots, \rho(n, k) \},
\]
while the maximum will occur at a unit vector with nonzero coordinates, $(n-1)$ of which are equal, that is, of the form
\[
(a_1,\dots,a_n) = (s,\underbrace{t,\dots,t}_{n-1\ \text{times}}),
\]
with $t\leq s$ and $(s,t)$ explicitly determined by $k$ and $n$ as a root of an explicit polynomial, see the Remark below for details.
\end{theorem}
In the following remarks we explain in details the behavior of the minimizer and the maximizer respectively.

\begin{remark}
We fix some $k$. In order to find the minimum of $\rho(s,k)$, consider the function $g:(0,\infty)\to(0,\infty)$,
\begin{equation}\label{eq:g-product}
  g(n)=\frac{\Gamma(n+k)}{n^{k/2}\,\Gamma(n)}=\frac{\prod_{j=0}^{k-1}(n+j)}{n^{k/2}}.
\end{equation}
Differentiating the logarithm with respect to $n$ gives
\[
  \frac{g'(n)}{g(n)}
  =\frac{d}{dn}\bigl(\ln g(n)\bigr)
  =\sum_{j=0}^{k-1}\frac{1}{n+j}-\frac{k}{2}\,\frac{1}{n}:=h(n).
\]
Since $g(n)>0$ for all $n>0$, the sign of $g'(n)$ coincides with the sign of $h(n)$, and the critical points of $g$ in $(0,\infty)$ are exactly the zeros of $h$.

To analyse $h$, first multiply by $n>0$:
\[
  n\,h(n)
  =\sum_{j=0}^{k-1}\frac{n}{n+j}-\frac{k}{2}
  =\sum_{j=0}^{k-1}\frac{1}{1+\frac{j}{n}}-\frac{k}{2}.
\]
Introduce the new variable
\(
  x=\frac{1}{n}>0
\)
and define
\[
  F(x)=\sum_{j=0}^{k-1}\frac{1}{1+jx}-\frac{k}{2}.
\]
Then
\[
  h(n)=0 \quad\Longleftrightarrow\quad F\!\left(\frac{1}{n}\right)=0.
\]
The derivative of $F$ is
\[
  F'(x)=\sum_{j=0}^{k-1}\frac{d}{dx}\left(\frac{1}{1+jx}\right) =-\sum_{j=1}^{k-1}\frac{j}{(1+jx)^2}.
\]
For $k\geq 2$ and $x>0$ every term in the last sum is negative, hence
\[
  F'(x)<0 \qquad \text{for all } x>0 \text{ and } k\geq 2.
\]
Thus $F$ is strictly decreasing on $(0,\infty)$ whenever $k\geq 2$.

The limits of $F$ at $0^+$ and $+\infty$ are easily computed. One has
\[
  \lim_{x\to 0^+}F(x)
  =\sum_{j=0}^{k-1}1-\frac{k}{2}
  =k-\frac{k}{2}
  =\frac{k}{2},
\]
and, since for $j\geq 1$ we have $1+jx\to\infty$ as $x\to\infty$,
\[
  \lim_{x\to\infty}F(x)
  =1-\frac{k}{2}.
\]
Combining these observations with the monotonicity of $F$ leads to the following conclusions about the equation $F(x)=0$ (equivalently $h(n)=0$).
For $k=2$ one has
\[
  F(0^+)=1, \qquad \lim_{x\to\infty}F(x)=0,
\]
and $F$ is strictly decreasing. Therefore $F(x)>0$ for all $x>0$, again implying that $h(n)$ has no zero in $(0,\infty)$.

For $k\geq 3$ one has
\[
  F(0^+)=\frac{k}{2}>0, \qquad \lim_{x\to\infty}F(x)=1-\frac{k}{2}<0,
\]
and $F$ is strictly decreasing on $(0,\infty)$. Hence, by the intermediate value theorem, there exists a unique $x_k>0$ such that $F(x_k)=0$. This implies that there is a unique $n_k>0$ with
\[
  \frac{1}{n_k}=x_k, \qquad h(n_k)=0.
\]
Therefore, for every integer $k\geq 3$, the function $g$ has exactly one critical point $n_k$ in $(0,\infty)$.

For $k\geq 3$ the expression \eqref{eq:g-product} can be rewritten as
\[
  g(n)=\frac{n(n+1)\cdots(n+k-1)}{n^{k/2}}.
\]
As $n\to 0^+$, the factors $n+1,\dots,n+k-1$ tend to positive constants, so the behaviour is dominated by
\[
  g(n)\sim C_k\,n^{1-k/2}
\]
for some constant $C_k>0$. Since $1-k/2<0$ for $k\geq 3$, one has $g(n)\to\infty$ as $n\to 0^+$. As $n\to\infty$, the product $n(n+1)\cdots(n+k-1)$ behaves like $n^k$, hence
\[
  g(n)\sim n^{k-k/2}=n^{k/2}\to\infty
\]
as $n\to\infty$. Together with the fact that $g'$ changes sign only once (because $h$ has exactly one zero), this shows that for $k\geq 3$ the function $g$ is strictly decreasing on $(0,n_k)$, strictly increasing on $(n_k,\infty)$, and attains a unique global minimum at $n=n_k$.

The preceding discussion also ensures that, for fixed $k$ and $n$, the maximum of $\rho(s,k)$, for $s=1,\ldots,n$, is attained either at $\rho(1,k)$ or at $\rho(n,k)$.

It is also useful to obtain an asymptotic approximation for the location of this minimum when $k$ is moderately large. The defining equation for $n_k$ is $h(n_k)=0$, that is
\[
  \sum_{j=0}^{k-1}\frac{1}{n_k+j}=\frac{k}{2n_k}.
\]
Approximating the sum by an integral gives
\[
  \sum_{j=0}^{k-1}\frac{1}{n+j}
  \approx \int_0^k \frac{dx}{n+x}
  =\ln\frac{n+k}{n}
  =\ln\Bigl(1+\frac{k}{n}\Bigr).
\]
Thus, for $n=n_k$, one expects approximately
\[
  \ln\Bigl(1+\frac{k}{n_k}\Bigr) \approx \frac{k}{2n_k}.
\]
Introducing the ratio
\[
  u=\frac{k}{n_k}
\]
this becomes the transcendental equation
\[
  \ln(1+u)=\frac{u}{2},
\]
which no longer involves $k$. This equation has a unique positive solution $u_0$, and a simple numerical computation shows that
\[
  u_0 \approx 2.51.
\]
Consequently,
\[
  n_k \approx \frac{k}{u_0} \approx 0.40\,k
\]
for large $k$. In other words, the location of the continuous minimizer grows asymptotically linearly in $k$ with slope slightly below $0.4$.

For concrete values of $k$, one can solve the equation $h(n)=0$ numerically. The following table lists, for $k=5,\dots,15$, an approximation of the unique minimizer $n_k$ in $(0,\infty)$ together with its integer part $\lfloor n_k\rfloor$.
\end{remark}
\begin{table}[H]
\centering
\begin{tabular}{ccc}
\toprule
$k$ & $n_k$ (approx.) & $\lfloor n_k \rfloor$ \\
\midrule
 5  & 1.2900 & 1 \\
 6  & 1.6958 & 1 \\
 7  & 2.0989 & 2 \\
 8  & 2.5006 & 2 \\
 9  & 2.9014 & 2 \\
10  & 3.3015 & 3 \\
11  & 3.7012 & 3 \\
12  & 4.1005 & 4 \\
13  & 4.4997 & 4 \\
14  & 4.8986 & 4 \\
15  & 5.2974 & 5 \\
\bottomrule
\end{tabular}
\caption{Approximate continuous minimizer $n_k$ of $g(n)$ and its integer part for various values of $k$.}
\end{table}
\begin{remark}
Let $x=s/t\ls1$. The maximizing configuration occurs as a root of the polynomial (see the proof below)
 \begin{align*}
g(x) &=  \binom{k}{1} \Gamma(n) \Gamma(k ) + \sum_{j=0}^{k-2} x^{j+1} \bigg[
\binom{k}{j} \Gamma(j + n-1) \Gamma(k - j + 1) (n-1) (j - k) \\
&\qquad + (j + 2)  \binom{k}{j+2} \Gamma(j +n+1) \Gamma(k - j - 1)
\bigg]  - \binom{k}{k-1} \Gamma(k+ n-2) 2 (n-1) x^k.
\end{align*}
Moreover, let $f(x)$ be the logarithm of $$\E\left(s\Gamma(n-1)+t\Gamma(1) \right)^{k}=\frac{\E\left(x\Gamma(n-1)+\Gamma(1) \right)^k}{((n-1)x^2+1)^{k/2}}.$$

There are inherent limitations in giving an exact description of the maximizer. In some cases it occurs at $x=1$ (see Figure~\ref{n=k=7}), whereas in other cases it may occur at one of the two additional roots of $g$ (see Figures~\ref{n=7 k=8} and \ref{n=7 k=10}).

\end{remark}

\begin{figure}[h]
\centering
\begin{minipage}{0.48\textwidth}
    \centering
    \begin{tikzpicture}
        \begin{axis}[
            width=6cm,
            xlabel={},
            ylabel={},
            domain=0:1,
            samples=400,
            grid=both,
        ]
            \addplot[thick]
                {ln( (3991680*x^7
                      +2328480*x^6
                      +1270080*x^5
                      +635040*x^4
                      +282240*x^3
                      +105840*x^2
                      +30240*x
                      +5040)
                    * (1/(6*x^2+1))^(3.5)
                  )};
        \end{axis}
    \end{tikzpicture}
    \caption{Plot of $f(x)$ for $n=k=7$.}
    \label{n=k=7}
\end{minipage}
\hfill
\begin{minipage}{0.48\textwidth}
    \centering
    \begin{tikzpicture}
        \begin{axis}[
            width=6cm,
            xlabel={},
            ylabel={},
            domain=0:1,
            samples=400,
            grid=both,
        ]
            \addplot[thick]
                {ln(
                    (362880*x^8
                    +2903040*x^7
                    +18627840*x^6
                    +10886400*x^5
                    +5443200*x^4
                    +2358720*x^3
                    +846720*x^2
                    +241920*x
                    +40320)
                    * (1/(6*x^2+1))^4
                )};
        \end{axis}
    \end{tikzpicture}
    \caption{Plot of $f(x)$ for $n=7$, $k=8$.}
    \label{n=7 k=8}
\end{minipage}
\end{figure}

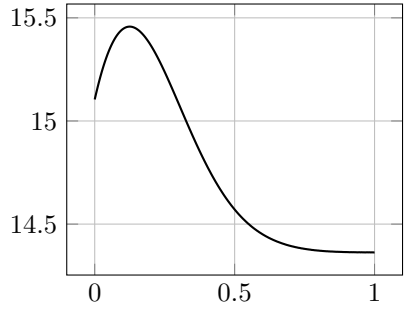
\begin{figure}[h]
    \centering
    \begin{tikzpicture}
        \begin{axis}[
            width=6cm,
            xlabel={},
            ylabel={},
            domain=0:1,
            samples=400,
            grid=both,
        ]
            \addplot[thick]
                {ln(
                    (10897286400*x^10
                    +7264857600*x^9
                    +4670265600*x^8
                    +2874009600*x^7
                    +1676505600*x^6
                    +914457600*x^5
                    +457228800*x^4
                    +203212800*x^3
                    +76204800*x^2
                    +21772800*x
                    +3628800)
                    * (1/(6*x^2 + 1))^5
                )};
        \end{axis}
    \end{tikzpicture}
    \caption{Plot of $f(x)$ for $n =7$ and $k=10$}
    \label{n=7 k=10}
\end{figure}

We also conjecture that this behavior holds for every real $q > 0$, as expressed in the following conjecture:

\begin{conjecture} \label{Conj1}
   Let $q > 0$ be a real number, and let $a_1, \ldots, a_n$ be non-negative real numbers satisfying $\sum_{i=1}^na_i^2= 1$.  If $X_1, \ldots, X_n$ are i.i.d.\ standard exponential random variables, then

    \[
    \mathbb{E}\big(a_1X_1+\cdots+a_nX_n \big)^q\gr \min\{ \rho(1,q),\ldots, \rho(n,q) \}.
    \]
\end{conjecture}

\subsection{Characterization of Extrema}
We first introduce a Lemma that will be useful for the characterization of the global extrema.

\begin{lemma}\label{a,b,c lemma}
Let $x, y, z$ and $a, b, c$ be non-negative real numbers such that $x + y + z = a + b + c$, $x^2 + y^2 + z^2 = a^2 + b^2 + c^2$, and $xyz \ls abc$. Then, for any integer $k \gr 1$, we have
$$
x^k + y^k + z^k \le a^k + b^k + c^k.
$$
Respectively, if $xyz \gr abc$, then
$$
x^k + y^k + z^k \gr a^k + b^k + c^k.
$$

\end{lemma}
\begin{proof}[Proof of Lemma \ref{a,b,c lemma}] 
Let $a$, $b$ and $c$ be pairwise distinct numbers and set $a+b+c=u$, $ab+ac+bc=v$ and $abc=w$.

Since $a^k+b^k+c^k$ is a symmetric polynomial, it can be expressed as $$ a^k + b^k + c^k := f (u, v, w),$$ and we must show that $f$ increases as a function of $w$.
For this, it is enough to show that $\frac{\partial f}{\partial w}\geq0.$

Computing the partial derivatives we get \begin{align*}1&=\frac{\partial(a+b+c)}{\partial u}=\frac{\partial a}{\partial u}+\frac{\partial b}{\partial u}+\frac{\partial c}{\partial u},\\
0=\frac{\partial(ab+ac+bc)}{\partial u}&=\frac{\partial a}{\partial u}b+\frac{\partial b}{\partial u}a+\frac{\partial a}{\partial u}c+\frac{\partial c}{\partial u}a+\frac{\partial b}{\partial u}c+\frac{\partial c}{\partial u}b\\
&=(b+c)\frac{\partial a}{\partial u}+(a+c)\frac{\partial b}{\partial u}+(a+b)\frac{\partial c}{\partial u}.
\end{align*}
Moreover, $$0=\frac{\partial(abc)}{\partial u}=bc\frac{\partial a}{\partial u}+ac\frac{\partial b}{\partial u}+ab\frac{\partial c}{\partial u}.$$
The determinant of this system equals
$$\Delta=\left|\left(\begin{array}{ccc} 1&1&1\\b+c&a+c&a+b\\bc&ac&ab\end{array}\right)\right|=\sum_{cyc}(ab(a+c)-bc(a+c))=
(a-b)(a-c)(b-c),$$
which gives $$\frac{\partial a}{\partial w}=\frac{1}{(a-b)(a-c)}.$$
Similarly, $$\frac{\partial b}{\partial w}=\frac{1}{(b-a)(b-c)}$$ and
$$\frac{\partial c}{\partial w}=\frac{1}{(c-a)(c-b)}.$$

To this end, note that $$\frac{\partial f}{\partial w}=\sum_{cyc}\frac{\partial f}{\partial a}\frac{\partial a}{\partial w}=\sum_{cyc}\frac{ka^{k-1}}{(a-b)(a-c)}=k\sum_{cyc}\frac{a^{k-1}}{(a-b)(a-c)}=kh_{k-3}(a,b,c)\geq 0,$$ 
for all $k\geq 3$, where in the last equality we used \eqref{lagr_pol}.

Since the above proof is valid for $a\rightarrow b^+$ and for $b\rightarrow c^+$ and since $f$ is a continuous function, we obtain that $f$ increases for any non-negative $a$, $b$ and $c$. 
\end{proof}

\begin{proposition} \label{Lagrange integers positive}
Let $n\gr1$ and $k>2$ be a  non-negative integer. Then the extrema of $h_k$ in the  the space $\mathbb{S}_{+}^{n-1}$ are of the form $$\boldsymbol{x}=\left(\underbrace{a, \ldots, a}_{\gamma_1 }, \underbrace{b, \ldots, b}_{\gamma_2 }\right).$$ Here, $ a $ appears $ \gamma_1 $ times, $ b $ appears $\gamma_2 $ times, subject to the constraints  $\gamma_1 a^2 + \gamma_2 b^2 = 1$ and $\gamma_1 + \gamma_2 = n$.  

Moreover, if $a \gr b>0$, the global minimum is attained at a vector of the form
$$
\boldsymbol{x} = \left( \underbrace{a, \ldots, a}_{n-1 }, b \right),
$$
where $(n-1)a^2+b^2=1$, while the global maximum is attained at a vector of the form
$$
\boldsymbol{x} = \left( \underbrace{b, \ldots, b}_{n-1}, a \right),
$$
where $(n-1)b^2+a^2=1$ , or where the extrema are attained at a vector with some components equal to zero and the remainder equal; specifically,
\[
\left(\underbrace{\frac{1}{\sqrt{\gamma}},\ldots,\frac{1}{\sqrt{\gamma}}}_{\gamma},\underbrace{0,\ldots,0}_{n-\gamma} \right),
\]
for $\gamma=1,\ldots,n$.
\end{proposition}
\begin{proof}
 We follow verbatim the proof of Proposition~\ref{Lag integers unconditional} for the case of extrema lying in the interior; unlike before, we now have a boundary. Since all coefficients are positive, the argument terminates at relation~\eqref{Lag d-3}, where for distinct \(x_i, x_j, x_k\), we have \(h_{d-3}(\boldsymbol{x}, x_i, x_j, x_k) = 0\), which is clearly a contradiction. The extrema on the boundary will occur when at least one coordinate is zero, then our argument shows that the remaining non-zero coordinates must be equal completing the first part of the proof. 

 For the second part, without loss of generality, we may assume that $a > b>0$. (If $b=0$ we get exactly the third described form.) Suppose, for the sake of contradiction, that the vector attaining the global minimum is not of the desired form, and assume that there exist at least two occurrences of $b$.  A similar argument provides the maximum.

Define the set
\[
S := \left\{ (x,y,z) \in \mathbb{R}_{+}^3 : x + y + z = 2b + a \quad \text{and} \quad x^2 + y^2 + z^2 = 2b^2 + a^2 \right\}.
\]
It is clear that $S$ is non-empty and compact. Thus, we can choose $x, y, z$ in $S$ such that the product $xyz$ is minimized. 

Observe that $xyz < b^2 a$ and the inequality is strict since Corollary~1.8 in~\cite{cirtoaje2007equal} suggests that the minimum cannot be attained at $(a,b,b)$. 

Applying Lemma~\ref{a,b,c lemma} we obtain that for every non-negative integer $m>2$
\[
p_m\left( a,\ldots,b,b \right) > p_m\left( x,a,\ldots,a,y,z\right),
\]
where the right-hand side is obtained from the vector $(a,\dots,b,b)$ by replacing one occurrence of $a$ with $x$, and by replacing  two occurrences of $b$ with $y$ and $z$, respectively.

Therefore, by the representation~\eqref{CHP-PSP}, we arrive at $h_k\left( a,\ldots,a,b,b \right)>h_k\left( x,a,\ldots,a,y,z\right),$ 
which is clearly a contradiction.
\end{proof}

\subsection{Proofs}

\begin{proof}[Proof of Theorem \ref{Schur-gamma}]
We view the moment generating function of $\sum_{j=1}^n\sqrt{x_j}X_j$ in two ways, using on the one hand independence and on the other the Taylor expansion of the exponential function, namely
\begin{align*}
\E e^{t\sum_{j=1}^n\sqrt{x_j}X_j} &= \sum_{k=1}^\infty \frac{t^k}{k!}\E\left(\sum_{j=1}^n\sqrt{x_j}X_j\right)^k\\
                                  &= \prod_{j=1}^n (1-t\sqrt{x_j})^{-1},
\end{align*}
for all $t$ such that the above is well defined. If we let $F_t(x_1,\ldots, x_n) = \prod_{j=1}^n (1-t\sqrt{x_j})^{-1}$, differentiating we get
\[
\frac{\partial F}{\partial x_i} = F\cdot\frac{t}{2\sqrt{x_i}(1-t\sqrt{x_i})}
\]
for every $i=1,\ldots,n$, which leads to
\begin{align*}
\left(\frac{\partial}{\partial x_j}-\frac{\partial}{\partial x_i}\right)F &= F\cdot\frac{\sqrt{x_j}-\sqrt{x_i}}{2\sqrt{x_ix_j}}\frac{t^2(\sqrt{x_j}+\sqrt{x_i})-t}{(1-\sqrt{x_i}t)(1-\sqrt{x_j}t)}\\
&= \prod_{k\neq i,j}(1-t\sqrt{x_k})\cdot\frac{\sqrt{x_j}-\sqrt{x_i}}{2\sqrt{x_ix_j}}\frac{t^2(\sqrt{x_j}+\sqrt{x_i})-t}{(1-\sqrt{x_i}t)^2(1-\sqrt{x_j}t)^2},
\end{align*}
for every $i\neq j$. {Taylor expanding} around 0, we see that
\[
\frac{(a+b)t^2-t}{(1-at)^2(1-bt)^2}=-t-(a+b)t^2-(a^2+b^2)t^3-(a-b)^2(a+b)t^4+O(t^5).
\]
In particular, the coefficient of $t^k$ is non-positive for any $a,b>0$ when $k\ls 4$. The wanted statement follows then by the Schur-Ostrowski criterion.
\end{proof}

\begin{remark}
It is worth to note that the above argument works only for $k\ls 4$. For example the Taylor coefficient of $t^5$ in the expansion of $\frac{(a+b)t^2-t}{(1-at)^2(1-bt)^2}$ is $-a^4+2a^3b+3a^2b^2+2ab^3-b^4$, which does not preserve sign, e.g. is positive when $a=b$.
\end{remark}

\begin{proof}[Proof of Corollary \ref{Cor Gamma pos k}]
This is a direct application of the function to the majorization sequence \eqref{maj seq}.
\end{proof}

\begin{proof}[Proof of Theorem \ref{pos k > n+2}]
   The importance of Proposition~\ref{Lagrange integers positive} lies in the fact that it reduces the problem to one involving only two variables, namely the study of $\E(aG_1+bG_2)^k$ when $\gamma_1a^2+\gamma_2b^2=~1$. We can further reduce this to a one-variable problem by dividing by $(\gamma_1 a^2 + \gamma_2 b^2)^{k/2}$ and factoring out a common factor of $b$. Thus, we are ultimately led to study the following:
   
For any $a, b > 0$, let $x := a / b$, and define the function $f : (0, \infty) \to \mathbb{R}$ by
\begin{align*}
f(x) &= \log\left( \sum_{j=0}^k \binom{k}{j} \Gamma(j + \gamma_1) \Gamma(k - j + \gamma_2) x^j \right) 
- \log\left( \Gamma(\gamma_1)\Gamma(\gamma_2) \right)- \frac{k}{2} \log(\gamma_1 x^2 + \gamma_2).
\end{align*}
its derivative can be computed as:
\[
f'(x) = \frac{g(x)}{(\gamma_1 x^2 + \gamma_2) \left( \sum_{j=0}^k \binom{k}{j} \Gamma(j + \gamma_1) \Gamma(k - j + \gamma_2) x^j \right)},
\]
where
\begin{align*}
g(x) &= \gamma_2 \binom{k}{1} \Gamma(1 + \gamma_1) \Gamma(k - 1 + \gamma_2) + \sum_{j=0}^{k-2} x^{j+1} \bigg[
\binom{k}{j} \Gamma(j + \gamma_1) \Gamma(k - j + \gamma_2) \gamma_1 (j - k) \\
&\qquad + (j + 2) \gamma_2 \binom{k}{j+2} \Gamma(j + 2 + \gamma_1) \Gamma(k - j - 2 + \gamma_2)
\bigg]  - \binom{k}{k-1} \Gamma(k - 1 + \gamma_1) \Gamma(1 + \gamma_2) \gamma_1 x^k.
\end{align*}
Observe that $g(0) > 0$ and $\lim_{x \to +\infty} g(x) = -\infty$.

We aim to better understand the coefficient in front of $x^{j+1}$, say $z_j$, which can be simplified as
\[
z_j := -\frac{1}{k - j - 1}(k - j - 1 + \gamma_2)(k - j - 2 + \gamma_2)\gamma_1 + \frac{1}{j + 1} \gamma_2 (j + 1 + \gamma_1)(j + \gamma_1),
\]
by factoring out
\[
\frac{k! \, (j + \gamma_1 - 1)! \, (k - j - 3 + \gamma_2)!}{j! \, (k - j - 2)!}.
\]
This can be rewritten as
\begin{align*}
\begin{split}
 \frac{1}{(j + 1)(k - j + 1)} \Big[ 
& -j^3 (\gamma_1 + \gamma_2) + j^2 \left( k(2\gamma_1 + \gamma_2) - 4\gamma_1 - 2\gamma_2 \right) - j \big( k^2 \gamma_1 + k(-5\gamma_1 - \gamma_2) + \gamma_1^2 \gamma_2 \\
&\qquad + \gamma_1(\gamma_2^2 - 2\gamma_2 + 5) + \gamma_2 \big)- \gamma_1 \big( k^2 + k(-\gamma_1 \gamma_2 + \gamma_2 - 3)+ \gamma_1 \gamma_2 + \gamma_2^2 - 2\gamma_2 + 2 \big)
\Big].
\end{split}
\end{align*}
According to Descartes' Rule of Signs, the number of positive real roots of a polynomial is at most equal, or is less than it by an even number, to the number of sign changes in the sequence of its coefficients. Since the numerator of $z_j$ is a cubic polynomial, it can exhibit up to three sign changes, and therefore,  in the worst-case scenario $g$ can have  three sign changes in the sequence of its coefficients.

Combining all the above, we conclude that $g$ has at most three positive roots.

We now claim that $f'(1) = 0$, and thus the obvious root of the polynomial $g$ is $x = 1$. To verify this, we compute the derivatives of $\E[(xG_1+G_2)^k]$. 
\begin{comment}
From ~\eqref{CHS by mgf means} we have
\[
\sum_{k=0}^{\infty} h_k(\underbrace{x,\ldots,x}_{\gamma_1}, \underbrace{1,\ldots,1}_{\gamma_2}) t^k = \left( \frac{1}{1 - tx} \right)^{\gamma_1} \left( \frac{1}{1 - t} \right)^{\gamma_2},
\]
 it follows that
\begin{align*}
\sum_{k=0}^{\infty} \frac{d}{d x} h_k(\underbrace{x,\ldots,x}_{\gamma_1}, \underbrace{1,\ldots,1}_{\gamma_2}) t^k 
&= \gamma_1 t \left( \frac{1}{1 - tx} \right)^{\gamma_1 + 1} \left( \frac{1}{1 - t} \right)^{\gamma_2}
\\
&= \gamma_1 \sum_{k=0}^{\infty} h_k(\underbrace{x,\ldots,x}_{\gamma_1+1}, \underbrace{1,\ldots,1}_{\gamma_2}) t^{k + 1}.
\end{align*}
Therefore,
\[
\frac{d }{d x}h_k\left(\underbrace{x,\ldots,x}_{\gamma_1}, \underbrace{1,\ldots,1}_{\gamma_2}\right)=\gamma_1h_{k-1}\left(\underbrace{x,\ldots,x}_{\gamma_1+1}, \underbrace{1,\ldots,1}_{\gamma_2}\right)
\]
and thus
\[
\frac{d}{d x} \mathbb{E}[(xG_1 + G_2)^k]  = k \gamma_1 \mathbb{E}[(xG_3 + G_2)^{k - 1}],
\]
\end{comment}
By differentiting under the integral sign we obtain $$\frac{d}{dx}\E\left[(xG_1+G_2)^k\right]=k\E\left[(xG_1+G_2)^{k-1}G_1\right].$$ Moreover
\[
\mathbb{E}[(xG_1 + G_2)^{k-1}G_1] = \mathbb{E}\left(\sum_{j=0}^{k-1} \binom{k-1}{j} (xG_1)^{j+1} G_2^{k-1-j} \right) = \sum_{j=0}^{k-1} \binom{k-1}{j} x^{j+1}\mathbb{E}(G_1^{j+1})\, \mathbb{E}(G_2^{k-1-j}).
\]
The moment formula for Gamma distributions gives:
\[
\mathbb{E}(G_1^{j+1}) = \frac{\Gamma(\gamma_1 + j + 1)}{\Gamma(\gamma_1)} \quad \text{and} \quad \mathbb{E}(G_2^{k - 1 - j}) = \frac{\Gamma(\gamma_2 + k - 1 - j)}{\Gamma(\gamma_2)}.
\]
Therefore,
\[
\mathbb{E}[(xG_1 + G_2)^{k-1}G_1] = \sum_{j=0}^{k-1} \binom{k-1}{j} x^{j+1}\frac{\Gamma(\gamma_1 + j + 1)}{\Gamma(\gamma_1)} \cdot \frac{\Gamma(\gamma_2 + k - 1 - j)}{\Gamma(\gamma_2)}.
\]
This sum simplifies to
\[
\gamma_1 \cdot \mathbb{E}[(xG_3 + G_4)^{k-1}] 
\]
where $G_3 \sim \Gamma(\gamma_1 + 1)$, $G_4 \sim \Gamma(\gamma_2)$
where $G_3 \sim \Gamma(\gamma_1 + 1)$ and is independent of $G_1,G_2$. In particular, for $x = 1$, either expression above can be evaluated directly. Similarly, we obtain the second derivative to be 
\[
\frac{d^2}{d x^2} \mathbb{E}[(xG_1 + G_2)^k]  = k(k-1) \gamma_1(\gamma_1-1) \mathbb{E}[(xG_4 + G_2)^{k - 1}],
\]
where $G_4\sim\Gamma(\gamma_1+2)$ and is independed from the others.
We observe also that
\[
\mathbb{E}[(G_1 + G_2)^k]  = \frac{\Gamma(n + k)}{\Gamma(n)}.
\]
Thus, putting everything together:
\[
f'(1)  = \gamma_1  \frac{k}{n} - \frac{k\gamma_1}{n} = 0.
\]

A similar argument as above yields the second derivative at $x = 1$:
\[
f''(1) = \frac{k\gamma_1}{n}  \frac{\gamma_2(k - n - 2)}{n(n + 1)}.
\]

Therefore, by the second-order derivative test, we conclude the following:
\begin{itemize}
    \item If $k > n + 2$, then $f''(1) > 0$, so $f$ has a local minimum at $x = 1$.
    \item If $k < n + 2$, then $f''(1) < 0$, so $f$ has a local maximum at $x = 1$.
\end{itemize}
Since $g(0) > 0$ and $ g(+\infty) < 0$, we conclude that for $k > n + 2$, the other two roots $s<t$ must be positioned as $0<s<1<t$ and thus the function $f$ may attain its minimum at one of the points $x = 0$, $x = 1$, or $x = +\infty$. That is, the possible candidates for the minimum of $f$ are $f(0)$, $f(1)$, and $f(+\infty)$. The maximum will be attained at one of $f(s),f(t)$.

In the case where $k = n + 2$,  $f$ has a root of multiplicity two at $x = 1$. Therefore, the possible minima occur at $f(0)$ and $f(+\infty)$.

In the case where $k < n + 2$, there are two possible configurations. In the first scenario, the global maximum is attained at $x = 1$, and the potential minima are at $f(0)$ and $f(+\infty)$. In the second scenario, the polynomial $g$ admits two distinct roots $s$ and $t$ such that $0 < s < t < 1$ (or $1<s<t)$. In this case, the maximum of $f$ occurs at $x = s$ or $x=1$ and  the minimum could potentially occur at $f(t)$, in addition to $f(0)$ and $f(+\infty)$. Both scenarios are possible so this argument cannot work.

For $k \leq n+1$, we employ a different approach. Recall  (see Proposition \ref{Lagrange integers positive}) that the global minimum will be attained at a vector of the form
\[
\boldsymbol{x} = \left( \underbrace{a,\ldots,a}_{n-1}, b \right),
\] 
for $a\gr b>0$ or at a vector of the form
\[
\left(\underbrace{\frac{1}{\sqrt{\gamma}},\ldots,\frac{1}{\sqrt{\gamma}}}_{\gamma},\underbrace{0,\ldots,0}_{n-\gamma} \right)
\]
where $\gamma=1,\ldots,n$. We proceed by induction on $k$ to show that for every natural number $n$ and $x \geq 1$, the following inequality holds:
\[
\frac{\mathbb{E}\left( x G_1 + G_2 \right)^k}{\left((n-1)x^2 + 1\right)^{k/2}} \geq \frac{\Gamma(n-1+k)}{\Gamma(n-1) \, (n-1)^{k/2}},
\]
where $G_1,G_2$ are independent $\Gamma(n-1)$ and $\Gamma(1)$ random variables respectively. 
This would complete the argument. We observe that the cases $k=1,2$ hold, so we assume the statement holds for $k-1$. 

Consider the function 
\[
h(x) = \frac{\mathbb{E}\left( x G_1 + G_2 \right)^k}{\left((n-1)x^2  + 1\right)^{k/2}}.
\]
As before, we compute its derivative:
\[
h'(x) = \frac{ k (n-1) \, \mathbb{E}\left( x G_3 + G_2 \right)^{k-1} \left( (n-1)x^2  + 1 \right)^{k/2} - k (n-1) x \, \mathbb{E}\left( x G_1 + G_2 \right)^k \left( (n-1)x^2  + 1 \right)^{k/2 - 1} }{ \left( x^2 (n-1) + 1 \right)^k },
\]
where $G_1,G_2$ and $G_3$ are independent Gamma random variables with distributions $G_1 \sim \Gamma(n-1)$, $G_2 \sim \Gamma(1)$ and $G_3\sim \Gamma(n)$, respectively.

Therefore, at any root $y$ of $h'(x)$ we have
\begin{align*}
h(y) 
&= \frac{\mathbb{E}\left( y G_1 + G_2 \right)^k}{\left( (n-1)y^2 + 1 \right)^{k/2}} \\
&= \frac{\mathbb{E}\left( y G_3 + G_2 \right)^{k-1}}{ \left( n y^2 + 1 \right)^{(k-1)/2} } \cdot \frac{ \left( n y^2 + 1 \right)^{(k-1)/2} }{ y \left( (n-1)y^2 + 1 \right)^{k/2 - 1} }.
\end{align*}

Notice that the function 
\[
g(y) = \frac{ \left( n y^2 + 1 \right)^{(k-1)/2} }{ y \left( (n-1) y^2 + 1 \right)^{k/2 - 1} },
\]
has derivative 
\[
g'(y)=\frac{(y^2(k-n-1)-1)(ny^2+1)^{\frac{k-3}{2}}}{y^2((n-1)y^2+1)^{k/2}},
\]
and thus is decreasing for $k \leq n+1$. Therefore,
\[
g(y) \gr \frac{ n^{(k-1)/2} }{ (n-1)^{k/2-1} }.
\]
Combining this estimate with the inductive hypothesis completes the proof. 
\end{proof}

To support Conjecture \ref{Conj1}  concerning real exponents, we provide a proof for the two-dimensional case.
\begin{proposition}\label{q n=2 pos}
Let $q$ be a non-negative real number and $a,b\gr0$ such that $a^2+b^2=1$. Then  
\[
\mathbb{E}(aX+bY)^q\geq \min \left \{\mathbb{E}\bigg(\frac{X+Y}{\sqrt{2}}\bigg)^q ,\mathbb{E}\big(X_1^q\big) \right\}.
\]
\end{proposition}

For the regime $q\ls4$ we will need the following log-concavity Lemma (see for example \cite{brazitikos2014geometry}).
\begin{lemma}\label{lem:Borell}
If $f:(0,+\infty)\to(0,+\infty)$ is log-concave then
\[
G(q) = \frac{1}{\Gamma(1+q)}\displaystyle\int_0^\infty t^qf(t)\,dt
\]
is also log-concave on $(-1,+\infty)$.
\end{lemma}
\begin{proof}[Proof of Proposition \ref{q n=2 pos}]
We will first work on the regime $q \leq 4$. We apply Lemma \ref{lem:Borell} for the density of $aX+bY$, given by
\[
 d(x)=\frac{1}{a - b} \left( e^{-x/a} - e^{-x/b} \right) \mathbf{1}_{(0, +\infty)}.
\]  

By log-concavity,
\[
\frac{1}{\E X_1^q}\E\left(aX+bY\right)^q = G(q)\gr G(0)^{1-\frac{q}{4}}G(4)^{\frac{q}{4}} = \left(\frac{\E\left(aX+bY\right)^4}{4!}\right)^\frac{q}{4}\gr 1,
\]
since, by Corollary \ref{Cor Gamma pos k}, $\E\left(aX+bY\right)^4\gr \E X_1^4=4!\,$.

For $q\gr4$, the desired quantity becomes  
\[
\frac{a^{q+1} - b^{q+1}}{(a - b)(a^2+b^2)^{q/2}} 
\]  
for all $a, b \gr 0$. Without loss of generality let $b\neq0$, $x:=a/b$ and define $f:(0,\infty)\to \mathbb{R}$ by
\[
f(x)=\log\left( \frac{x^{q+1}-1}{x-1}\right)-\frac{q}{2}\log(x^2+1).
\]
 To this end, we examine the  monotonicity of $f$. 
\[
f'(x)= \frac{qx^{q+1}-(q+1)x^q+1}{(x-1)(x^{q+1}-1)}-\frac{qx}{x^2+1}
\]
\[
=\frac{- x^{q + 2} +q x^{q + 1}  -  x^q(q+1)  +  x^2(q+1) - q x  + 1}{(x-1)(x^2+1)(x^{q+1}-1)}
\]
 Note that the  numerator, say $g$, has  five sign changes in its coefficients, therefore by the extension of Descartes' rule of signs \cite{curtiss1918recent} it has at most five positive roots, and thus the same will hold for $f'$.
 It is easy to check that $g(x)$ has a root of multiplicity three at $x=1$.

We observe that $g'''(1)=q(q+1)(q-4)>0$ for $q>4$ and thus from the higher order derivative test, $x=1$ is a saddle point and a strictly increasing point of inflection. 
\end{proof}

\section{The Centred Case: Proofs}

\subsection{Characterization of Extrema}
For the \textbf{centred} case, where $\sum_{i=1}^n a_i = 0$ and $\sum_{i=1}^n a_i^2 = 1$, we have the following theorem.

When $n$ is even, the lower bound for every even degree coincides with the one stated in Theorem~\ref{Hunter conj}, but in this case we are able to determine the exact upper bound.

\begin{theorem}\label{h_{2r} lower geom}
Let $n$ be an even integer and $k \ge 1$ be an integer, and let  $a_1, \ldots, a_n$ be real numbers such that $\sum_{i=1}^n a_i^2 = 1$ and $\sum_{i=1}^na_i=0$. Then $$h_{2k}(a_1, \ldots, a_n) \geq \frac{\left(\frac{n}{2} + k - 1\right)!}{k! \cdot \left(\frac{n}{2} - 1\right)! \cdot n^k},$$ and equality is attained at the vector $\tilde{a}$. 

 The maximum, for both odd and even non-negative integers $n$, is attained at the vector where all of $a_i$ except of one are equal, that is 
 \[
 \left(-\underbrace{\frac{1}{\sqrt{n(n-1)}},\ldots,-\frac{1}{\sqrt{n(n-1)}}}_{n-1},\sqrt{\frac{n-1}{n}} \right).
 \]

\end{theorem}
We also provide a proof of the lower bound for $h_4$ for both odd and even integers $n$.
\begin{proposition}\label{h_4 geo}
Let $a_1, \ldots, a_n \in \mathbb{R}$ satisfy $\sum_{i=1}^n a_i = 0$ and $\sum_{i=1}^n a_i^2 = 1$. Then the quantity $h_4(a_1, \ldots, a_n)$ attains the minimum at
$$\left(\underbrace{\frac{1}{\sqrt{n}},\ldots,\frac{1}{\sqrt{n}}}_{\frac{n}{2}},\underbrace{-\frac{1}{\sqrt{n}},\ldots,-\frac{1}{\sqrt{n}}}_{\frac{n}{2}}\right)$$ if $n$ is even, or
$$\left(\underbrace{\sqrt{\frac{n+1}{n(n-1)}},\ldots,\sqrt{\frac{n+1}{n(n-1)}}}_{\frac{n-1}{2}},\underbrace{-\sqrt{\frac{n-1}{n(n+1)}},\ldots,-\sqrt{\frac{n-1}{n(n+1)}}}_{\frac{n+1}{2}} \right),$$
if $n$ is odd.
\end{proposition}

In the case $n = 3$, we are able to obtain sharp upper and lower bounds for nearly all moments. 
\begin{proposition}\label{geom n=3}
Let 
\[
x_1 = \left(\frac{1}{\sqrt{6}}, \frac{1}{\sqrt{6}}, -\frac{2}{\sqrt{6}}\right), \quad
x_2 = \left(\frac{1}{\sqrt{2}}, -\frac{1}{\sqrt{2}}, 0\right),
\]
and define 
\[
f(a_1, a_2, a_3) := \mathbb{E} \left| a_1 X_1 + a_2 X_2 + a_3 X_3 \right|^q,
\]
subject to the constraints $a_1 + a_2 + a_3 = 0$ and $a_1^2 + a_2^2 + a_3^2 = 1$.

\begin{enumerate}
   \item[(i)] For $q \in (-1, 0) \cup (2, 4)$, the function $f$ attains its minimum at $x_1$ and its maximum at $x_2$.
\item[(ii)] For $q \in (0, 2)\cup (4,6)$, the function $f$ is minimized at $x_2$ and maximized at $x_1$.
\item[(iii)] For $q = k > 3$, where $k$ is a non-negative integer, the function $f$ is minimized at $x_2$ and maximized at $x_1$.
\item[(iv)] Moreover, for $q = 2$ and $q = 4$, the function $f$ is constant.
\end{enumerate}
\end{proposition}

Before proceeding with the proofs, we first characterize the extrema in the geometric case, that is $\sum_{i=1}^n a_i=0$ and $\sum_{i=1}^n a_i^2=1$. Following almost verbatim the proof of Proposition \ref{Lag integers unconditional} we obtain the following:
\begin{proposition} \label{Lag integers geometric}
    Let $n\gr1$ and  $d>3$ be a even integer. Then the extrema of $h_d$ on the space $\mathcal{S}:=\{ (x_1,\ldots,x_n)\in \mathbb{R}^n: \sum_{i=1}^n x_i=0, \sum_{i=1}^n x_i^2=1 \} $ are of the form $$\boldsymbol{x}=\left(\underbrace{a, \ldots, a}_{\gamma_1 }, \underbrace{b, \ldots, b}_{\gamma_2 },\underbrace{c,\ldots,c}_{\gamma_3 
 }\right).$$ Here, $ a $ appears $ \gamma_1 $ times, $ b $ appears $\gamma_2 $ times, and $ c $ appears $ \gamma_3 $ times, subject to the constraints constraints $\gamma_1 a^2 + \gamma_2 b^2 + \gamma_3 c^2 = 1$ and $\gamma_1 + \gamma_2 + \gamma_3 = n$.
\end{proposition}

\subsection{Sharp Bounds for $p_m$}
\begin{proposition}\label{p_k max}
Let $n \ge 2$, and let $a, b, c$ be real numbers such that 
\[
\gamma_1 a + \gamma_2 b + \gamma_3 c = 0 
\quad \text{and} \quad 
\gamma_1 a^2 + \gamma_2 b^2 + \gamma_3 c^2 = 1,
\]
and set
\[
\boldsymbol{x}
= \left(
\underbrace{a, \ldots, a}_{\gamma_1},\,
\underbrace{b, \ldots, b}_{\gamma_2},\,
\underbrace{c, \ldots, c}_{\gamma_3}
\right).
\]
 Then, for every positive integer $m$, $p_m\left(\boldsymbol{x} \right)=\gamma_1a^m+\gamma_2b^m+\gamma_3c^m$ attains its maximum at
 \[
 \left(\underbrace{-\frac{1}{\sqrt{n(n-1)}},\ldots,-\frac{1}{\sqrt{n(n-1)}}}_{n-1},\sqrt{\frac{n-1}{n}} \right),
 \]
 that is, the vector in which all but one coordinate are equal.

\end{proposition}
\begin{proof}
Since $n=2$ is trivial, assume that $n\gr3$, $\gamma_1,\gamma_2,\gamma_3\gr1$ and $m\gr3$ since $p_1(\boldsymbol{x})=0$, $p_2(\boldsymbol{x})=1$. We will distinguish two cases. If $m$ is odd, say $m=2k+1$, then using the Lagrange multiplier method, we find that an extreme point $(a,b,c)$ satisfies 
$$(2k+1)x^{2k}-2\lambda x-\mu=0,$$ for $x=a,b,c$. Since $x^{2k}$ is convex, it has at most two points of intersection with a affine function, therefore, two of $a,b,c$ are equal, say $a=b$. Then, if we set $\gamma_1+\gamma_2=\gamma$ we obtain
$$p_m=\gamma a^m+\gamma_3c^m$$ under the conditions $\gamma a+\gamma_3 c=0$ and $\gamma a^2+\gamma_3 c^2=1$. Solving for $a,c$ we and using $\gamma_3=n-\gamma$, we write
$$p_m=\frac{1}{n^{m/2}}\left(\frac{(n-\gamma)^{m/2}}{\gamma^{m/2-1}}-\frac{\gamma^{m/2}}{(n-\gamma)^{m/2-1}}\right),$$ which is clearly maximized for $\gamma=1$.
\bigskip

If $m$ is even, say $m=2k$, then using the Lagrange multiplier method we find that the vector $(a,b,c)$ which attains the maximum satisfies
\begin{equation}\label{lag_main}
2k\,x^{2k-1}-2\lambda x-\mu=0,\end{equation} for $x=a,b,c$. We can multiply this relation by $\gamma_1, \gamma_2,\gamma_3$, respectively, and add them to obtain
$kp_{2k}=~\lambda$. We multiply this relation by $\gamma_1a, \gamma_2b,\gamma_3c$, respectively, and add them to get
$2kp_{2k-1}=~n\mu$. If substitute this into \eqref{lag_main}, and divide by $2m$, 
$$x^{2k-1}=x\,p_{2k}+\frac{p_{2k-1}}{n}.$$ 
Due to the even power, suppose, without loss of generality, that $a$ is positive and $|c|>a>|b|$. Under this assumption, we first prove that
$$\frac{1}{\sqrt{n(n-1)}}\leq a\leq\frac{1}{\sqrt{2}}.$$
For the upper bound, note that 
$$2a^2\leq (\gamma_1+\gamma_3)a^2\leq \gamma_1a^2+\gamma_2b^2+\gamma_3c^2=1.$$
For the lower bound we write 
$$(\gamma_1+\gamma_2)a^2+\frac{a^2(\gamma_1+\gamma_2)^2}{\gamma_3}\geq\gamma_1a^2+\gamma_2b^2+\frac{(\gamma_1a+\gamma_2b)^2}{\gamma_3}=1,$$ or 
$$a^2\geq\frac{\gamma_3}{n(n-\gamma_3)}\geq\frac{1}{n(n-1)}.$$
From the odd case, we know that 
$$p_{2k-1}\geq -A_{n,2k-1}:= - \left[\left(\frac{n - 1}{n}\right)^{(2k - 1)/2} +(-1)^{2k-1}(n - 1)\left(\frac{1}{(n - 1)n}\right)^{(2k - 1)/2}\right].$$
On the other hand, since $(a,b,c)$ is the point that achieves the maximum we have that 
$$p_{2k}\geq A_{n,2k}.$$
Therefore,
from the main relation for $a$ satisfies the inequality 
$$a^{2k-1}\geq a\, A_{n,2k}-\frac{A_{n,2k-1}}{n}.$$
We consider the following function:
\[
f_k(x) =x^{2k-1}- x\, A_{n,2k}+\frac{A_{n,2k-1}}{n}.
\]
and the constants:
\[
s = \sqrt{\frac{n-1}{n}}, \quad
t = \frac{1}{\sqrt{n(n-1)}}, \quad
w = \frac{1}{\sqrt{2}}.
\]
Notice that from Descartes's rule or signs, $f_k$ has at most 2 positive roots.

We aim to prove that \( f_k(w) < 0 \) for every $n\gr 3$ and $k\gr2$. The conclusion is that $f_k$ has two positive roots, one of them is $t$ and the other is larger than $w$, since the limit of $f_k$ as $x$ grows to $+\infty$ is $+\infty$.  This means that in the open interval $(t,w)$ $f_k$ is negative. Therefore, $a=t$. In all the above inequalities, equality must hold, therefore, $\gamma_3=1$ and $a=b$.

Let us  begin with  \( k \geq 2 \) and $ n \geq 6$. 

We express \( f_k(w) \) as:
\begin{align*}
f_k(w) &= (w -t) \left[
 w^{2k-2} + w^{2k-3} t + \cdots + t^{2k-2}-A_{n,2k}
\right]\\
&:=\left(w-t\right)Q_k(w).
\end{align*}
Since the prefactor \(w -t > 0 \), for $n\geq 3$ the sign of \( Q_k(w) \) is the same as that of \( f_k(w) \).
If we set
\[q:=\frac{t}{w}<1
\]
then
\[
S(w):=\sum_{i=0}^{2k-2}t^i\,w^{2k-2-i}
=w^{2k-2}\sum_{i=0}^{2k-2}q^i
<\frac{w^{2k-2}}{1-q},
\]
which gives 
\[
\frac{S(w)}{s^{2k}}
<\frac{1}{1-q}\,\frac{w^{2k-2}}{s^{2k}}
=\frac{w^{-2}}{1-q}\Bigl(\frac{w^2}{s^2}\Bigr)^k\leq \frac{w^{-2}}{1-q}\Bigl(\frac{w^2}{s^2}\Bigr)^2=\frac{w^2}{s^4(1-q)}<1,
\]
since $s^4\geq (5/6)^2$ and $q\ls \frac{1}{\sqrt{15}}$, for $n\geq 6$. This means
\[
S(w)-\bigl(s^{2k}+(n-1)t^{2k}\bigr)<0.
\]

Finally, we can easily check that the above argument implies that $f_k(w)<0$ also holds for $n\geq 4$ and $k\geq 3$. Indeed,
For $n=5$ we have $s^2=4/5$ and $q=t/w=\sqrt{1/10}<1/3$, which implies
\[
\frac{w^2}{s^4(1-q)}
\le \frac{1/2}{(16/25)\cdot (2/3)}
= \frac{75}{64}<1,
\]
and the same conclusion follows. For $n=4$ we retain the factor $(w^2/s^2)^k$ and use that
\[
\frac{S(w)}{s^{2k}}
<\frac{w^{-2}}{1-q}\left(\frac{w^2}{s^2}\right)^k
=\frac{2}{1-q}\Bigl(\frac{2}{3}\Bigr)^k
\]
with $q=1/\sqrt6$. It is then easy to check that the right-hand side is $<1$ for all $k\ge4$, hence $f_k(w)<0$ in this range. Finally, in the remaining case $(n,k)=(4,3)$ we compute explicitly
\[
f_3(w)
=\frac{20\sqrt3-25\sqrt2}{288}<0.
\]
In particular, we have $f_k(w)<0$ for all $n\ge4$ and $k\ge3$.

Therefore, it remains  to check only the cases $n=3$ and any $k\gr2$, $(n,k)=(5,2)$ and $(n,k)=(4,2)$. The cases for $k=2$ have been already settled, as it can be seen in Proposition 2.1 of \cite{cirtoaje2015equal}. For $n=3$ and $k\gr2$ it suffices to consider the quantity $$\frac{a^{2k}+b^{2k}+c^{2k}}{(a^2+b^2+c^2)^k}.$$ Since $a+b+c=0$ and $a^2+b^2+c^2=1$, we can without loss of generality assume that $a,c$ are non-negative, $c\neq0 $ and set $x=a/c\gr0$ and $y=b/c$ to obtain a function of one variable $f:(0,\infty)\rightarrow \R$ as follows:
\[
f(x)=\frac{x^{2k}+(x+1)^{2k}+1}{2^k(x^2+x+1)^k}.
\]
We will prove that $f$ attains its maximum at $x=1$.
Let $g$ be the numerator of $f'$, then a direct computation shows that 
\[
\frac{g(x)}{kx}=x^{2k+1}+3x^{2k}+2x^{2k-1}+(1-x)(x+1)^{2k}-2x^2-3x-1
\]
Since $k$ is a integer we can expand the binomial and by setting $a_j:=\binom{2k}{j}$ we obtain
\begin{align*}
\frac{g(x)}{kx^2} =\;& 
  -x^{2k-1}(-3 - a_{2k} + a_{2k-1}) 
  - x^{2k-2}(-2 - a_{2k-1} + a_{2k-2}) \\
&+ \sum_{j=2}^{2k-3} (a_{j+1} - a_j)x^j + x(-2 + a_2 - a_1) 
  + (-3 + a_1 - a_0).
\end{align*}
Since $a_j = a_{2k-j}$, it follows immediately that $\frac{g(x)}{kx^2}$ is an anti-palindromic polynomial. In particular, $x = 1$ is a root. The uniqueness of the positive root then follows from the sign of  the sequence $a_{j+1}-a_j=a_j\frac{2k-2j-1}{j+1}$, together with Descartes' Rule of Signs.

\end{proof}

\begin{proof}[Proof of Theorem \ref{h_{2r} lower geom}]
For the lower bound, applying  Theorem~\ref{Hunter conj}  we conclude:
\[
h_{2r}(\boldsymbol{x}) \geq \frac{(n/2 + r - 1)!}{r! \cdot (n/2 - 1)! \cdot n^r}.
\]

For the maximum, notice that relation \eqref{CHP-PSP} suggest $$h_k(\boldsymbol{x}) \ls \sum_{m_1 + 2m_2 + \cdots + km_k = k \atop m_1\ge 0, \ldots, m_k\ge 0} \prod_{i=1}^k \frac{|p_i|^{m_i}(\boldsymbol{x})}{m_i ! \, i^{m_i}}.$$ 
Notice also that \( p_{2k} \geq 0 \). Moreover, if \( \boldsymbol{w} \) is a maximizer of \( p_{2k-1} \), then since \( p_{2k-1} \) is an odd function, it follows that \( -\boldsymbol{x} \) is a minimizer. Thus,
\[
|p_{2k-1}| \leq p_{2k-1}(\boldsymbol{w}).
\]
The  characterization of the extrema Proposition \ref{Lag integers geometric} and Proposition \ref{p_k max} now finishes the proof.
\end{proof}

\begin{proof}[Proof of Proposition \ref{h_4 geo}]
We will find upper and lower bounds for 
$$\mathbb{E}\left|\sum_{i=1}^n a_iX_i\right|^4=3+6\sum_{i=1}^n a_i^4,$$
under the conditions $\sum_{i=1}^n a_i=0$ and $\sum_{i=1}^n a_i^2=1$. 
We will use the method of Lagrange multipliers to find the maximum and the minimum which exist since the domain is compact. We are searching for all $\boldsymbol{a}=(a_1,\ldots,a_n)$ and the real numbers $\lambda,\mu$ such that for all indices $i=1,\ldots, n$ we have
$$4a_i^3-2\lambda a_i+\mu=0.$$ Multiplying by $a_i$ and summing for all $i$ we get 
\begin{equation}\label{lagr_1}4\sum_{i=1}^n a_i^4-2\lambda=0.
\end{equation}If for the indices $i\neq j$ it is true that $a_i\neq a_j$, then subtracting the two relations we get 
\begin{equation}\label{lagr_2}
4(a_i^2+a_ia_j+a_j^2)-2\lambda=0.\end{equation}
Therefore if there exists one more index $k$, such that $a_k\neq a_i$ and $a_k\neq a_j$, then $a_i+a_j+a_k=0$. Therefore the vector $\boldsymbol{a}$ has the form $$(x,x,\ldots,x,y,y,\ldots,y,z,z,\ldots, z),$$ where $x+y+z=0$. Suppose that $x$ appears $s$ times, $y$ appears $t$ times and $z$ appears $w$ times and due to the symmetry we assume that $s\geq t\geq w$. From \eqref{lagr_2} and \eqref{lagr_1}, it suffices to find upper bounds for $x^2+y^2+z^2$, under the conditions $$\begin{cases}
x+y+z=0&\\
sx+ty+wz=0&\\
sx^2+ty^2+wz^2=1
\end{cases}$$
If $s=t=w$, then $x^2+y^2+z^2=\frac{3}{n}.$
Otherwise, we can solve the above system with respect to $x,y,z$ and find that 
$$x^2+y^2+z^2=\frac{(s-t)^2+(t-w)^2+(s-w)^2}{w(s-t)^2+s(t-w)^2+t(s-w)^2}.$$
We will prove the following double inequality.
$$\frac{2}{n-1}\leq\frac{(s-t)^2+(t-w)^2+(s-w)^2}{w(s-t)^2+s(t-w)^2+t(s-w)^2}.$$
 For the left-hand side inequality, note that $n=w+s+t$ and after some algebraic manipulations we end up proving that 
$$(w-1)(w-s)(w-t)+(t-1)(t-s)(t-w)+(s-1)(s-w)(s-t)\geq 0.$$ To this end, note that the last one is equivalent to 
$$(s-t)[(s-1)(s-w)-(t-1)(t-w)]+(w-1)(s-w)(t-w) \geq 0,$$
which is true, since $w,s,t\geq 1$ and $s\geq t\geq w$.

We also need to consider the case where exactly one group is zero, say $z$. In this case, since $x + y = 0$ and also $sx + ty = 0$, together with the condition $sx^2 + ty^2 = 1$, we are led to $s = t$, and a vector of the form
\[
\left(
\underbrace{\frac{1}{\sqrt{2s}}, \ldots, \frac{1}{\sqrt{2s}}}_{s\text{-times}},
\underbrace{-\frac{1}{\sqrt{2s}}, \ldots, -\frac{1}{\sqrt{2s}}}_{s\text{-times}},
\underbrace{0, \ldots, 0}_{(n - 2s)\text{-times}}
\right).
\]
Note also that $2s = n - w$. Since we only need to bound $x^2= \frac{1}{2s}$, which is decreasing in $s$, it suffices to consider the case $2s = n - 1$ for the minimum and $s=1$ for the maximum.

It remains to check the case where the vector $\boldsymbol{a}$ has the form $$(x,x,\ldots,x,y,y,\ldots,y),$$
where $x$ appears $s$ times and $y$ appears $t$ times. Without loss of generality, assume that $x>0$ and $y<0$. Then, 
from \eqref{lagr_2} and \eqref{lagr_1}, it suffices to find upper the maximum and the minimum for $x^2+xy+y^2$, under the conditions $$\begin{cases}
sx+ty=0&\\
sx^2+ty^2=1
\end{cases}$$
Solving the system we get that 
$$x^2+y^2+xy=\frac{1}{n}\left(\frac{t}{s}+\frac{s}{t}-1\right)=\frac{1}{n}\left(\frac{n-s}{s}+\frac{s}{n-s}-1\right).$$
The function in the parenthesis takes its maximum value for $s=1$ (or $s=n-1$) and its minimum value for $s=\lfloor\frac{n}{2}\rfloor$. It is easy now to compare the extrema.
\end{proof}

\begin{proof}[Proof of Proposition \ref{geom n=3}]
For (iv), we observe that since $a_1 + a_2 + a_3 = 0$ and $a_1^2 + a_2^2 + a_3^2 = 1$, it follows that
\[
\mathbb{E}\left|a_1 X_1 + a_2 X_2 + a_3 X_3\right|^2 = 1,
\]
and
\[
\mathbb{E}\left|a_1 X_1 + a_2 X_2 + a_3 X_3\right|^4 = 3 + 6(a_1^4 + a_2^4 + a_3^4) = 6.
\]

For (i), (ii) and (iii), a simple point-wise bound for $\Re\left(\phi_{\sum a_i X_i}\right)$, combined with the Fourier formulas \eqref{Fourier (0,2)}, \eqref{Fourier (2,4)} and \eqref{Fourier (4,6)} suffices.

We observe that
\[
\phi(t) = \frac{1}{(1 + i a_1 t)(1+ia_2t)(1+ia_3t)}  
= \frac{1 - i t (a_1 + a_2 + a_3) - t^2(a_1 a_2 + a_2 a_3 + a_1 a_3) - i^3t^3 a_1 a_2 a_3 }{(1 + a_1^2 t^2)(1 + a_2^2 t^2)(1 + a_3^2 t^2)}.
\]

Since $a_1 + a_2 + a_3 = 0$ and $a_1^2 + a_2^2 + a_3^2 = 1$, this simplifies to
\[
\Re(\phi(t)) = \frac{1 + t^2/2}{(1 + a_1^2 t^2)(1 + a_2^2 t^2)(1 + a_3^2 t^2)}.
\]

We now establish a bound for $\Re(\phi(t))$:
\[
(1 + t^2/2)^2 \leq (1 + a_1^2 t^2)(1 + a_2^2 t^2)(1 + a_3^2 t^2) \leq (1 + t^2/6)^2 (1 + 2t^2/3),
\]
where the upper bound is attained at $x_2$ and the lower bound at $x_1$.

This can be proven by expanding the product:
\[
(1 + a_1^2 t^2)(1 + a_2^2 t^2)(1 + a_3^2 t^2) = 1 + t^2(a_1^2 + a_2^2 + a_3^2) + t^4(a_1^2 a_2^2 + a_2^2 a_3^2 + a_1^2 a_3^2) + t^6 a_1^2 a_2^2 a_3^2.
\]

The lower bound is immediate. For the upper bound, we may assume without loss of generality that two of the variables share the same sign; let $a_1, a_2 \gr 0$.

Applying the Cauchy–Schwarz inequality yields
\[
1 = a_1^2 + a_2^2 + a_3^2 \geq 2\left(\frac{a_1 + a_2}{2}\right)^2 + a_3^2,
\]
which implies $a_3^2 \leq \frac{2}{3}$.

Using the AM–GM inequality, we obtain
\[
a_1^2 a_2^2 a_3^2 \leq \left(\frac{a_1 + a_2}{2}\right)^4 a_3^2 = \frac{a_3^6}{2^4} \leq \frac{(2/3)^3}{16} = \frac{1}{54}.
\]
For integers $k > 3$, a different approach is required.
We define the function
$$
f(a_1,a_2,a_3)=\frac{\mathbb{E}|a_1X_1+a_2X_2+a_3X_3|^{q}}{(a_1^2+a_2^2+a_3^2)^{q/2}}.
$$
Let $a_1,a_2\gr0$, $a_2\neq 0$, and $a_3\ls0$, and write $-a_3\ls0$, so that the expression becomes $a_1X_1+a_2X_2-a_3X_3$. 
Let $x=a_1/a_2$ and $y=a_3/a_2$, then it equals
$$
\frac{\mathbb{E}|xX_1+X_2-yX_3|^{q}}{(x^2+y^2+1)^{q/2}}.
$$
Since $a_1+a_2-a_3=0$ and $a_1^2+a_2^2+a_3^2=1$, it follows that
$$
f(x)=\frac{\mathbb{E}|xX_1+X_2-(1+x)X_3|^{q}}{2^{q/2}(x^2+x+1)^{q/2}}.
$$
We now  utilize  formula \eqref{interpol} and drop the $\frac{\Gamma(1+q)}{2^{q/2+1}}$ constant for simplicity:
$$
f(x)= \frac{-1 - 2 x + 2 x^{2 + q} + x^{3 + q} + ( x-1) (1 + x)^{2 + q}}{(x-1) (x+2) (2x+1)(x^2+x+1)^{q/2}}.
$$

As a motivating example, when $q=7$, we compute:
$$
f'(x)=\frac{-3 x (10 x^9 + 80 x^8 + 249 x^7 + 363 x^6 + 183 x^5 - 183 x^4 - 363 x^3 - 249 x^2 - 80 x - 10)}{2 (x^2 + (1 + x))^{9/2} (1 + 2 x)^2 (2 + x)^2}.
$$
The function $f'$ has a unique root at $x=1$ by Descartes' rule of signs.  

\begin{claim}\label{Clm2}
Let $q=k$ be a non-negative integer. The numerator of $f'$, denoted $g$, has a root at $x=1$ and exhibits exactly one sign change in its coefficients. More generally, $g$ is an anti-palindromic polynomial with a single sign change in its coefficients.
\end{claim}

\begin{proof}[Proof of Claim \ref{Clm2}]
From expanding $(x+1)^{k+2}$, we get
$$
f(x)=\frac{2x^{k+2}+2+\sum_{j=1}^{k+1}x^j\left(\binom{k+2}{j}+3\right)}{(x^2+x+1)^{k/2}(2x+1)(x+2)}.
$$
Direct computation shows:
\begin{align*}
g(x) &= (2x^4+7x^3+9x^2+7x+2)\left[2(k+2)x^{k+1}+ \sum_{j=1}^{k+1} jx^{j-1}\left(\binom{k+2}{j}+3\right)\right] \\
&\quad - \frac{1}{2}( x^3(4k+8) +  x^2(12k+18) +  x(9k+18) + 2 k + 10) \left(2x^{k+2}+2+\sum_{j=1}^{k+1}x^j\left(\binom{k+2}{j}+3\right)\right).
\end{align*}

Define  $a_j=\binom{k+2}{j}+3$, we notice that the constant term, $x^{k+5}$, and $x^{k+4}$ vanish, and thus a direct computation shows:
\begin{align*}
\frac{g(x)}{x} &= 
a_1(2 - k) + 4a_2 - (9k + 18) \\
&\quad + x\left(-\tfrac{9}{2}k\,a_1 + (9 - k)a_2 + 6a_3 - (12k + 18)\right) \\
&\quad + x^2\left[ 
  -(4k + 8) 
  - (6k + 2)a_1 
  + \left(9 - \tfrac{9}{2}k\right)a_2 
  + (16 - k)a_3 
  + 8a_4 
\right] \\
&\quad + \sum_{j=1}^{k - 3} x^{j+2} \Big[
  (2j - 2k - 4)a_j 
  + (7j - 6k - 2)a_{j+1} 
  + \left(9j - \tfrac{9}{2}k + 9\right)a_{j+2} 
  + (7j-k+16)a_{j+3}+(2j+8)a_{j+4}
\Big] \\
&\quad + x^{k} \Big[
  (4k + 8)  
  + (6k + 2)a_{k+1}
  - \left(9 - \tfrac{9}{2}k\right)a_k
  - (16 - k)a_{k-1} 
  - 8a_{k-2} 
\Big] \\
&\quad + x^{k+1} \left[
  \tfrac{9}{2}k\,a_{k+1} 
  - (9 - k)a_{k}
  - 6a_{k-1} 
  + (12k + 18)
\right] \\
&\quad + x^{k+2} \left[
  -(2 - k)a_{k+1} 
  - 4a_k 
  + (9k + 18)
\right]
\end{align*}

Note that $a_j=a_{k+2-j}$, and all coefficients of $g(x)$ are symmetric except for a sign flip. So we write:
$$
g(x)=x\sum_{j=0}^{k+2}\beta_jx^j,
$$
and its easy to see that $\beta_j=-\beta_{k+2-j}$, i.e., $g(x)/x$ is anti-palindromic. 
Let $z:=z_{j,k}$ be the quantity inside the sum for $j$ and $w:=w_{j,k}$ for $k-j-2$. Then,
\[
z=a_{j}(2j-2k-4)+a_{j+1}(7j-6k-2)+a_{j+2}(9j-9/2k+9)+a_{j+3}(7j-k+16)+a_{j+4}(2j+8)
\]
and 
\[
w =a_{k-j-2}(-2j-8)+a_{k-j-1}(k-7j-16)+a_{k-j}(9/2k-9j-9)+a_{k-j+1}(6k-7j+2)+a_{k-j+2}(2k-2j+4)
\]
We now notice that $z=-w$, proving anti-palindromicity and thus confirming $x=1$ is a root. 

To ensure uniqueness of the root, we show the first $\lfloor(k+2)/2\rfloor$ coefficients have the same sign. 
Since we can easily deal we the terms outside of the sum, we want to prove that $z_{j,k}\gr0$ for $j=1,2,\ldots,\lfloor(k-2)/2\rfloor$.
Direct computations show that 
\[
z_{j,k}=\frac{(k+2)!(2j-k+2)(j^2+j(2-k)-2k^2-7k+3)}{2j!(k-j-2)!(j+1)(j+2)(j+3)(k-j+1)(k-j)(k-j-1)} +\frac{81}{2}(2j-k+2)
\]
\[
=\frac{(k-2j-2)}{2}\bigg(\frac{(k+2)!(2k^2+7k-j^2+j(k-2)-3)}{(j+3)!(k-j+1)!}-81\bigg).
\]
Since $k-2j-2\gr0$, it suffices to show that the second parenthesis is non-negative. Indeed, we have that for $k\geq 6$
$$2k^2+7k-j^2+j(k-2)-3\geq \frac{7k^2}{4}+8k-4>81$$
and 
$$\frac{(k+2)!}{(j+3)!(k-j+1)!}=\binom{k+2}{j+1}\frac{1}{(j+2)(j+3)}\geq \binom{k+2}{2}\frac{4}{(k+2)(k+4)}>1.$$
This inequality holds for $k\ge6$, and for $3< k\le5$ it can be checked directly.

Finally, we verify positivity of the remaining coefficients:
\begin{align*}
\beta_0 &= (k-2)(k-4) \ge 0, \\
\beta_1 &= \frac{1}{2}(k-2)(k-4)(k+9)\ge 0, \\
\beta_2 &= \frac{1}{12}(k-4)(2k^3+15k^2+49k-270) \ge 0.
\end{align*}
\end{proof}
By Descartes' rule of signs, $f'$ has a unique root with sign pattern $+,-$ so $x=0$ is a minimum and $x=1$ a maximum.
\end{proof}

\section{Minimum of $h_{2k}$ under the $\|x\|_{\infty}$ constraint}
In our setting, the sharp lower comparison between $\|A\|_{H_d}$ and the
operator norm $\|A\|_{\mathrm{op}}$ is controlled by the minimum of $h_d$ on
the $\ell_\infty$–sphere, since
\[
\|A\|_{H_d}^d
= h_d\big(s_1(A),\ldots,s_n(A)\big)
= \|A\|_{\mathrm{op}}^d\,h_d\!\left(
\frac{s_1(A)}{\|A\|_{\mathrm{op}}},\ldots,
\frac{s_n(A)}{\|A\|_{\mathrm{op}}}
\right),
\]
so the best constant $C_{n,d}$ in an inequality of the form
\[
C_{n,d}\,\|A\|_{\mathrm{op}}\le \|A\|_{H_d}
\]
is exactly the $d$-th root of $\min\{h_d(x):\|x\|_\infty=1\}$. In
Theorem~\ref{thm:min-h-2k-linfty} below we show that every non-vertex local
minimiser of $h_{2k}$ on the $\ell_\infty$–sphere has the very rigid form
$(t,\ldots,t,1)$, with $t$ determined by a one-variable polynomial equation.
This reduces the optimal comparison between $\|A\|_{H_d}$ and
$\|A\|_{\mathrm{op}}$ to a one-dimensional optimisation problem and
characterises the extremal matrices as those whose normalised singular-value
vector has this $(t,\ldots,t,1)$ structure; see
Corollary~\ref{cor:matrix-op-vs-Hd}.
\begin{theorem}
\label{thm:min-h-2k-linfty}
Let $n\ge2$ and $k\ge1$ be integers, and let
$
S_\infty^{n-1}:=\{x\in\mathbb{R}^n:\|x\|_\infty=1\}.$
The global minimum of $h_{2k}$ on $S_\infty^{n-1}$ is always attained at a point of the form $(t,\ldots,t,1)$ where $t$ is the unique root of $$\varphi_{k,n}(t)
        := h_{2k-1}(\underbrace{t,\ldots,t}_{n-1},1)$$
in $(-1,0)$. Moreover, the value of $h_{2k}$ at this vector equals 
$$\binom{n+2k-1}{2k}\,t^{2k}.$$
\end{theorem}

\begin{proof}
Since $h_{2k}$ is continuous and $S_\infty^{n-1}$ is compact, $h_{2k}$ attains a global minimum on $S_\infty^{n-1}$. Let $x^\ast=(x_1^\ast,\ldots,x_n^\ast)\in S_\infty^{n-1}$ be a global minimiser which is not a vertex of $\{\pm1\}^n$. Since $h_{2k}$ is invariant under permutations of the coordinates and under the global sign change $x\mapsto -x$, we may assume without loss of generality that \(
x_n^\ast = 1\). We will first show that 
$$|x_i^\ast|<1,$$ for $i=1,\ldots,n-1$. If this is not the case, then there exists some $j,k$, such that $|x_j^\ast|=1$ and $|x_k^\ast|=s<1$. Using Theorem \ref{thm:Tao} 
$$h_{2k}(x^\ast)> h_{2k}(y^\ast),$$ where $y^\ast$ has exactly the same coordinates as $x^\ast$, except $x_j^\ast,x_k^\ast$, where $y^\ast$ has their arithmetic mean. Since $y^\ast\in S_{\infty}^{n-1}$, we arrive at a contradiction.  
Thus $x^\ast$ lies in the relative interior of the $(n-1)$–dimensional face
\[
F := \{x\in[-1,1]^n : x_n=1\}.
\]
On the relative interior of $F$ the coordinates $x_1,\ldots,x_{n-1}$ are unconstrained, so the restriction of $h_{2k}$ to $F$ has vanishing gradient at $x^\ast$, i.e.
\[
\frac{\partial}{\partial x_i} h_{2k}(x^\ast) = 0,
\]
for all $i=1,\ldots,n-1$.
By Lemma~\ref{Hunter lemma} we have, for each $i$,
\[
\frac{\partial}{\partial x_i} h_{2k}(x)
    = h_{2k-1}(x,x_i),
\]
Hence
\[
h_{2k-1}(x^\ast,x_i^\ast)=0
\]
for all $i=1,\ldots,n-1$.
Fix distinct indices $i,j\in\{1,\ldots,n-1\}$. Using the difference identity~\eqref{CHS diff} with
\[
x=(x_1^\ast,\ldots,x_n^\ast),\quad a=x_i^\ast,\quad b=x_j^\ast,
\]
we obtain
\[
h_{2k-1}(x^\ast,x_i^\ast)-h_{2k-1}(x^\ast,x_j^\ast)
    =(x_i^\ast-x_j^\ast)\,h_{2k-2}(x^\ast,x_i^\ast,x_j^\ast).
\]
The left-hand side vanishes, so
\[
(x_i^\ast-x_j^\ast)\,h_{2k-2}(x^\ast,x_i^\ast,x_j^\ast)=0.
\]
By Hunter’s positivity theorem (Theorem~\ref{thm:hunter-positivity}), the even-degree polynomial $h_{2k-2}$ is strictly positive at every non-zero vector. The vector
\[
(x^\ast,x_i^\ast,x_j^\ast)
=(x_1^\ast,\ldots,x_n^\ast,x_i^\ast,x_j^\ast)
\]
is non-zero because $x_n^\ast = 1$, hence 
$h_{2k-2}(x^\ast,x_i^\ast,x_j^\ast)>0$.
Therefore we must have $x_i^\ast=x_j^\ast$ for all $1\le i,j\le n-1$. This shows that every non-vertex local minimiser has the form
\[
x^\ast=(t,\ldots,t,1)
\]
for some $t\in(-1,1)$.

\noindent To find the global minimiser we need to check the value of $h_{2k}$ at every vertex of the cube $\{\pm1\}^n$. However, if $(\epsilon_1,\ldots,\epsilon_{n-1})\in \{\pm 1\}^{n-1}$, using Theorem \ref{thm:Tao}, we obtain 
$$h_{2k}(\epsilon_1,\ldots,\epsilon_{n-1},1)\geq h_{2k}(A,\ldots,A,1),  $$
where $$A=\frac{\sum_{i=1}^{n-1}\epsilon_i}{n-1}\in [-1,1],$$ therefore the minimum is not attained at a vertex, except possibly the vertices $(1,\ldots, 1)$ and $(-1,\ldots,-1,1)$. We know that the global maximum is attained in the first one. For the latter, observe that Theorem \ref{thm:Tao} reassures that $h_{2k}$ achieves a smaller value at $(1,1,\ldots, 1,0,0)$. This observation completes the proof. The uniqueness of $t$ and the exact form of the minimum will follow from the lemmata below.  
\begin{lemma}\label{lem:closed-form-critical}
Fix integers $n\ge2$ and $k\ge1$. For $t\in\mathbb{R}$ and $m\in\mathbb{N}$ set
\[
H_m^{(n)}(t):=h_m(\underbrace{t,\ldots,t}_{n},1).
\]
Then
\begin{equation}\label{eq:explicit-Hm}
H_m^{(n)}(t)
=\sum_{j=0}^{m} \binom{n+j-1}{j}\,t^{j}.
\end{equation}
In particular, there exists a unique $t_{n,k}\in(-1,0)$ such that
\[
H_{2k-1}^{(n)}(t_{n,k})=h_{2k-1}(t_{n}[n],1)=0,
\]
and the corresponding interior critical point for the minimization of
$h_{2k}$ on the $\ell_\infty$–sphere $\{a\in\mathbb{R}^n:\|a\|_\infty=1\}$
is \(
a^\ast=(\underbrace{t_{n,k},\ldots,t_{n,k}}_{n-1},1).
\)
At this point we have the closed form
\begin{equation}\label{eq:value-at-critical}
h_{2k}(a^\ast)
=h_{2k}(\underbrace{t_{n,k},\ldots,t_{n,k}}_{n-1},1)
=\binom{n+2k-1}{2k}\,t_{n,k}^{2k}.
\end{equation}
\end{lemma}

\begin{proof}
The generating function representation
\[
\sum_{m=0}^\infty h_m(a_1,\ldots,a_{n+1})z^m
=\frac{1}{(1-a_1z)\cdots(1-a_{n+1}z)}
\]
with $a_1=\cdots=a_n=t$ and $a_{n+1}=1$ gives
\[
\sum_{m=0}^\infty H_m^{(n)}(t)\,z^m
=\frac{1}{(1-tz)^n(1-z)}.
\]
Expanding
\[
\frac{1}{(1-tz)^n}
=\sum_{j=0}^\infty \binom{n+j-1}{j} t^j z^j,
\qquad
\frac{1}{1-z}=\sum_{r=0}^\infty z^r
\]
and collecting the coefficient of $z^m$ yields \eqref{eq:explicit-Hm}.

We next relate $h_{2k}$ on $(t,\ldots,t,1)$ and on
$(t,\ldots,t,1,0)$.  Using the difference identity
$h_{m-1}(x,a)-h_{m-1}(x,b)=(a-b)\,h_{m-2}(x,a,b)$ from
Lemma~\ref{Hunter lemma} with $a=t$ and $b=0$ we obtain, for any $m\ge1$,
\[
h_m(t[n-1],1)
=h_m(t[n],1,0)-t\,h_{m-1}(t[n-1],1,t,0).
\]
By symmetry of $h_m$ this becomes
\begin{equation}\label{eq:lift-drop-zero}
h_m(\underbrace{t,\ldots,t}_{n-1},1)
=h_m(\underbrace{t,\ldots,t}_{n},1)
 -t\,h_{m-1}(\underbrace{t,\ldots,t}_{n},1).
\end{equation}

Let $t_{n,k}\in(-1,0)$ be a solution of the stationarity equation
\[
h_{2k-1}(t_{n,k}[n],1)=H_{2k-1}^{(n)}(t_{n,k})=0.
\]
(Existence and uniqueness of such a solution in $(-1,0)$ is stated in
Lemma~\ref{lem:root-unique} below.)  Plugging $m=2k$ and $t=t_{n,k}$
into \eqref{eq:lift-drop-zero} and using
$h_{2k-1}(t_{n,k}[n],1)=0$ gives
\begin{equation}\label{eq:drop-zero-at-root}
h_{2k}(t_{n,k}[n-1],1) = h_{2k}(t_{n,k}[n],1)
=H_{2k}^{(n)}(t_{n,k}).
\end{equation}

Finally, applying \eqref{eq:explicit-Hm} with $m=2k$ we have
\[
H_{2k}^{(n)}(t)
=\sum_{j=0}^{2k} \binom{n+j-1}{j}\,t^{j}
=H_{2k-1}^{(n)}(t)+\binom{n+2k-1}{2k}t^{2k}.
\]
Evaluating at $t=t_{n,k}$ and using
$H_{2k-1}^{(n)}(t_{n,k})=0$ yields
\[
H_{2k}^{(n)}(t_{n,k})=\binom{n+2k-1}{2k}t_{n,k}^{2k}.
\]
Combining this with \eqref{eq:drop-zero-at-root} gives
\eqref{eq:value-at-critical}.
\end{proof}

\begin{lemma}\label{lem:root-unique}
For each $n\ge1$ and $k\ge1$ there exists a unique $t_{n,k}\in(-1,0)$
such that
\[
h_{2k-1}(\underbrace{t_{n,k},\ldots,t_{n,k}}_{n},1)=0.
\]
The sequence
$\{t_{n,k}\}_{k\ge1}$ satisfies
\[
\lim_{k\to\infty}t_{n,k}=-1.
\]
\end{lemma}
\begin{proof}
Let $X_1,\dots,X_{n+1}$ be i.i.d.\ standard exponential random variables. By
the moment representation of complete homogeneous symmetric polynomials,
\[
\Phi(t)
= h_{2k-1}(\underbrace{t,\dots,t}_{n},1)
= \frac{1}{(2k-1)!}\,
\mathbb{E}\bigl( t(X_1+\cdots+X_n)+X_{n+1}\bigr)^{2k-1}.
\]
Set
\[
S:=X_1+\cdots+X_n,\qquad Y_t:=tS+X_{n+1},
\]
so that
\[
\Phi(t)=\frac{1}{(2k-1)!}\,\mathbb{E}\bigl(Y_t^{\,2k-1}\bigr).
\]

Since all moments of $Y_t$ are finite and $t\mapsto Y_t$ is affine, we may
differentiate under the expectation to obtain, for every $t\in\mathbb{R}$,
\[
\Phi'(t)
= \frac{1}{(2k-1)!}\,\mathbb{E}\!\left[(2k-1)S\,Y_t^{\,2k-2}\right]
= \frac{1}{(2k-2)!}\,\mathbb{E}\bigl[S\,Y_t^{\,2k-2}\bigr].
\]

Now fix $t\in(-1,0)$. Then $S>0$ almost surely, and $2k-2$ is even, so
$Y_t^{\,2k-2}\ge0$ almost surely. Moreover, the joint law of $(S,X_{n+1})$ is
absolutely continuous, hence
\[
\mathbb{P}\bigl(S=0\bigr)=\mathbb{P}\bigl(Y_t=0\bigr)=0,
\]
which implies
\[
S\,Y_t^{\,2k-2}>0 \quad\text{almost surely}.
\]
Therefore
\[
\Phi'(t)=\frac{1}{(2k-2)!}\,\mathbb{E}\bigl[S\,Y_t^{\,2k-2}\bigr]>0
\qquad\text{for all }t\in(-1,0),
\]
so $\Phi$ is strictly increasing on $(-1,0)$.

Consequently $\Phi$ can have at most one zero in $(-1,0)$. On the other hand,
\[
\Phi(0)=h_{2k-1}(0,\dots,0,1)=1>0,
\]
and a direct computation from the explicit formula for $h_{2k-1}(t,\dots,t,1)$
shows that $\Phi(-1)<0$. By the intermediate value theorem there exists
$t_\ast\in(-1,0)$ with $\Phi(t_\ast)=0$, and by strict monotonicity this zero
is unique. Since $\Phi'(t_\ast)>0$, the root is simple.

To understand the asymptotic behaviour, note that for $|\rho|<1$
\[
(1+\rho)^{-n}
=\sum_{j=0}^{\infty} (-1)^j
\binom{n+j-1}{j}\rho^j.
\]
Thus, for $0<\rho<1$ we may write
\[
(1+\rho)^{-n}
=P_{n,k}(\rho)+R_{n,k}(\rho),
\qquad
R_{n,k}(\rho)
:=\sum_{j=2k}^{\infty} (-1)^j
\binom{n+j-1}{j}\rho^j.
\]
At the root $\rho=\rho_{n,k}=-t_{n,k}$ we have $P_{n,k}(\rho_{n,k})=0$, and
hence
\begin{equation}\label{eq:root-tail-identity}
(1+\rho_{n,k})^{-n}=R_{n,k}(\rho_{n,k}).
\end{equation}

Fix $\varepsilon\in(0,1)$ and set $q:=1-\varepsilon$.  For
$0<\rho\le q$ and all $j\ge2k$ we use the crude bound
$\binom{n+j-1}{j}\le C_n j^{n-1}$ (for some constant $C_n$ depending
only on $n$) to obtain
\[
|R_{n,k}(\rho)|
\le\sum_{j=2k}^\infty \binom{n+j-1}{j}\rho^j
\le C_n\sum_{j=2k}^\infty j^{n-1} q^j
\le C'_n (2k)^{n-1} q^{2k},
\]
for a suitable constant $C'_n$.  The right-hand side tends to $0$ as
$k\to\infty$, uniformly in $0<\rho\le q$.  On the other hand
$(1+\rho)^{-n}\ge(1+q)^{-n}>0$ for all such $\rho$.  Therefore, for
$k$ sufficiently large the identity \eqref{eq:root-tail-identity}
cannot hold with $\rho_{n,k}\le q$.  Since $q<1$ was arbitrary, it
follows that $\rho_{n,k}\to1$ as $k\to\infty$, and hence
$t_{n,k}=-\rho_{n,k}\to-1$.
\end{proof}
\end{proof}

\begin{corollary}\label{cor:matrix-op-vs-Hd}
Let $d=2k$ be an even integer, and let $A\in M_n(\mathbb{C})$ have singular values
$s_1(A)\ge\cdots\ge s_n(A)\ge0$. Define the norm induced by the complete homogeneous
symmetric polynomial $h_d$ by
\[
\|A\|_{H_d} := h_d\big(s_1(A),\ldots,s_n(A)\big)^{1/d}.
\]
Then,
\[
\binom{n+2k-1}{2k}^{1/2k}\,|t_{n,k}|\,\|A\|_{\mathrm{op}}
\;\le\;
\|A\|_{H_{2k}}
\;\le\;
\binom{n+2k-1}{2k}^{1/2k}\,\|A\|_{\mathrm{op}},
\]
where $t_{n,k}$ is the unique number in $(-1,0)$ such that \[
h_{2k-1}(\underbrace{t_{n,k},\ldots,t_{n,k}}_{n},1)=0,
\]
and the bound is optimal.
Equality in the lower bound can occur
only if the normalised singular-value vector $s(A)/\|A\|_{\mathrm{op}}$ is a minimiser
of $h_d$ on $S_\infty^{n-1}$. 
\end{corollary}

\begin{proof}
Let $A\in M_n(\mathbb{C})$ and put $x=s(A)/\|A\|_{\mathrm{op}}$. Then
$\|x\|_\infty=1$ and, by definition,
\[
\|A\|_{H_d}^d = h_d\big(s_1(A),\ldots,s_n(A)\big)
= \|A\|_{\mathrm{op}}^d\, h_d(x).
\]
Taking the minimum and maximum of $h_d(x)$ over $\{x\in\mathbb{R}^n:\|x\|_\infty=1\}$
yields the stated inequality. Optimality is
immediate by evaluating at matrices whose normalised singular-value vector is a
minimiser of $h_d$ on $S_\infty^{n-1}$; the structural description of such
minimisers is given by Theorem~\ref{thm:min-h-2k-linfty}.
\end{proof}

\noindent

{\bf Acknowledgements} The  authors acknowledges support by the Hellenic Foundation for Research and Innovation (H.F.R.I.) under the call “Basic research Financing (Horizontal support of all Sciences)” under the National Recovery and Resilience Plan “Greece 2.0” funded by the European Union–NextGeneration EU (H.F.R.I. Project Number:15445).

\bibliographystyle{plain}
\bibliography{citations}

\end{document}